\documentclass[11pt]{amsart}
\usepackage{
amssymb
,amsthm
,amsmath
,mathtools
,mathdots
,mathrsfs
,url
,longtable
}
\usepackage[all]{xy}
\usepackage[top=35truemm,bottom=35truemm,left=30truemm,right=30truemm]{geometry}
\SelectTips{cm}{12}
\objectmargin={1pt}

\usepackage{tikz}
\usetikzlibrary{intersections,calc,arrows.meta}

\newcommand{\Z}{\mathbb{Z}}
\newcommand{\R}{\mathbb{R}}
\newcommand{\C}{\mathbb{C}}

\newcommand{\Rr}{\mathscr{R}}
\newcommand{\LL}{\mathscr{L}}
\newcommand{\EE}{\mathscr{E}}
\newcommand{\VV}{\mathscr{V}}
\newcommand{\Sy}{\mathfrak{S}}

\newcommand{\GL}{\mathrm{GL}}
\newcommand{\SL}{\mathrm{SL}}
\newcommand{\SO}{\mathrm{SO}}
\newcommand{\Sp}{\mathrm{Sp}}

\newcommand{\Cusp}{\mathscr{C}}
\newcommand{\Rep}{\mathrm{Rep}}
\newcommand{\Irr}{\mathrm{Irr}}
\newcommand{\temp}{\mathrm{temp}}
\newcommand{\Ind}{\mathrm{Ind}}
\newcommand{\Jac}{\mathrm{Jac}}
\newcommand{\supp}{\mathrm{supp}}
\newcommand{\sgn}{\mathrm{sgn}}
\newcommand{\id}{\mathrm{id}}

\newcommand{\iif}{&\quad&\text{if }}
\newcommand{\other}{&\quad&\text{otherwise}}
\newcommand{\resp}{resp.~}
\renewcommand{\1}{\mathbf{1}}
\newcommand{\ep}{\varepsilon}
\newcommand{\less}{\prec}
\newcommand{\s}{\mathfrak{s}}

\newcommand{\half}[1]{\frac{#1}{2}}
\newcommand{\ub}[1]{\underline{#1}}

\newtheorem{thm}{Theorem}[section]
\newtheorem{lem}[thm]{Lemma}
\newtheorem{prop}[thm]{Proposition}
\newtheorem{cor}[thm]{Corollary}
\newtheorem{rem}[thm]{Remark}
\newtheorem{defi}[thm]{Definition}
\newtheorem{ex}[thm]{Example}

\newenvironment{talign*}
{\csname align*\endcsname}
{\endalign}

\title{An analogue of ladder representations for classical groups}
\author{Hiraku Atobe}
\date{}
\address{
Department of Mathematics, Hokkaido University,
Kita 10, Nishi 8, Kita-Ku, Sapporo, Hokkaido, 060-0810, Japan 
}
\email{
atobe@math.sci.hokudai.ac.jp
}

\allowdisplaybreaks

\begin{document}
\maketitle

\begin{abstract}
In this paper, 
we introduce a notion of \emph{ladder representations} 
for split odd special orthogonal groups and symplectic groups 
over a non-archimedean local field of characteristic zero. 
This is a natural class in the admissible dual which contains both 
strongly positive discrete series representations 
and irreducible representations with irreducible $A$-parameters.
We compute Jacquet modules and the Aubert duals of ladder representations, 
and we establish a formula to describing ladder representations 
in terms of linear combinations of standard modules.
\end{abstract}

\section{Introduction}
Let $F$ be a non-archimedean local field of characteristic zero with the normalized absolute value $|\cdot|$.  
For $n \geq 0$, 
we set $G_n$ to be a split special orthogonal group $\SO_{2n+1}(F)$ or a symplectic group $\Sp_{2n}(F)$. 
Denote by $\Rr(\GL_n(F))$ (\resp $\Rr(G_n)$)
the Grothendieck group of the category of smooth representations of $\GL_n(F)$ (\resp $G_n$) of finite length. 
For an object $\pi$ of this category, 
we denote by $[\pi]$ the corresponding element of the Grothendieck group. 
It is known that $\Rr(\GL_n(F))$ (\resp $\Rr(G_n)$) has two natural bases. 
One consists of the \emph{irreducible representations} and the other consists of the \emph{standard modules}.
The change of basis matrix is triangular and unipotent (in an appropriate sense). 
At least for the general linear groups case, 
as predicted by Zelevinsky \cite{Z2,Z3} and proven by Chriss--Ginzburg \cite{CG}, 
this matrix is written using Kazhdan--Lusztig polynomials. 
It is very complicated to compute this matrix in general, 
but in a certain special class of irreducible representations $\pi$, 
it may be relatively simple to describing $[\pi]$ in terms of a linear combination of standard modules. 
\vskip 10pt

As an example, let us recall the \emph{ladder representations} of $\GL_n(F)$. 
In Introduction, we only consider unipotent representations.
Let 
\[
[x,y] \coloneqq \{|\cdot|^x, |\cdot|^{x-1}, \dots, |\cdot|^y\}
\]
be a segment with $x,y \in \R$ such that $x \geq y$ and $x-y \in \Z$. 
The parabolically induced representation 
\[
|\cdot|^x \times |\cdot|^{x-1} \times \dots \times |\cdot|^y
\]
of $\GL_{x-y+1}(F)$
has a unique irreducible subrepresentation, 
which is denoted by $\Delta[x,y]$ and is called a \emph{Steinberg representation}. 
For the notation of parabolically induced representations of $\GL_n(F)$, see \S \ref{rep0} below. 
We also write $\Delta[x,x+1] \coloneqq \1_{\GL_0(F)}$ and $\Delta[x,y] \coloneqq 0$ for $y > x+1$.
When segments $[x_1,y_1], \dots, [x_t,y_t]$ satisfy that 
$x_1+y_1 \leq \dots \leq x_t+y_t$, 
the parabolically induced representation
\[
\Delta[x_1,y_1] \times \dots \times \Delta[x_t,y_t]
\]
is called a \emph{standard module}. 
It has a unique irreducible subrepresentation, 
which we denote by $L(\Delta[x_1,y_1], \dots, \Delta[x_t,y_t])$. 
As a natural generalization of Steinberg representations, 
following Kret--Lapid \cite{KL} and Lapid--M\'inguez \cite{LM},
we define a (unipotent) \emph{ladder representation}
by an irreducible representation of the form $L(\Delta[x_1,y_1], \dots, \Delta[x_t,y_t])$
such that $x_1<\dots<x_t$, $y_1<\dots<y_t$ and $x_1 \equiv \dots \equiv x_t \bmod \Z$.
Moreover, we call $L(\Delta[x_1,y_1], \dots, \Delta[x_t,y_t])$ a \emph{(shifted) Speh representation}
if $x_{i+1} = x_i+1$ and $y_{i+1} = y_i+1$ for $i=1,\dots,t-1$.
\vskip 10pt

Let $\Sy_t$ be the symmetric group on $\{1,\dots,t\}$ with the sign character $\sgn \colon \Sy_t \rightarrow \{\pm1\}$. 
The following is called a \emph{determinantal formula}. 
\begin{thm}[Lapid--M{\'i}nguez {\cite[Theorem 1 (ii)]{LM}}]\label{det1}
Let $\pi = L(\Delta[x_1,y_1], \dots, \Delta[x_t,y_t])$ 
be a ladder representation of $\GL_n(F)$. 
Then in the Grothendieck group $\Rr(\GL_n(F))$, an equality
\[
[\pi] = \sum_{\sigma \in \Sy_t}\sgn(\sigma) 
\Big[\Delta[x_1,y_{\sigma(1)}] \times \dots \times \Delta[x_t,y_{\sigma(t)}]\Big]
\]
holds.
\end{thm}
When $\pi$ is a Speh representation, 
the determinantal formula was early proven by Tadi{\'c} \cite{T-det} 
and its proof was simplified by Chenevier--Renard \cite{CR}.
There are several applications for the determinantal formula. 
For example: 

\begin{itemize}
\item
The determinantal formula for Speh representations
was used in the critical final step of the global classification of Arthur's endoscopic classification 
(see \cite[\S 7.5, \S 8.2]{Ar}).

\item
Using Theorem \ref{det1}, 
Tadi\'c computed the full derivative (in the sense of Zelevinsky) of ladder representations 
(see \cite[Theorem 14]{LM}). 

\item
M\'inguez and S\'echerre established 
another determinantal formula to extend the mod $\ell$ Jacquet--Langlands correspondence 
to all irreducible representations (see \cite[Proposition 1.15, Th{\'e}or{\`e}me 1.16]{MS}).

\item
Theorem \ref{det1} was one of the key ingredients of 
the theory of local newforms for non-generic representations of $\GL_n(F)$
established by the author together with Kondo and Yasuda (see \cite[\S 6]{AKY}).
\end{itemize}
\vskip 10pt

Let us turn our attention to the classical group $G_n$. 
An analogue of Speh representations 
should be irreducible representations with \emph{irreducible $A$-parameters}. 
Such representations are unitary and can be realized as local components of 
discrete spectrum of automorphic forms.
For more precision, see \cite{Ar}.
However, as is the same as Speh representations, 
the class of irreducible representations of $G_n$ 
with (irreducible) $A$-parameters for all $n \geq 0$
is not preserved by derivatives (in the sense of Jantzen--M\'inguez).
Note that every ladder representation of $\GL_n(F)$ 
can be obtained from a Speh representation of $\GL_{n'}(F)$ for $n' \geq n$ by derivatives. 
Inspired by M{\oe}glin's construction of irreducible representations with $A$-parameters 
(see \cite{X2}), we will define a notion of \emph{ladder representations} of $G_n$
in Definition \ref{def:ladder}. 
Here, we only state the \emph{unipotent} ladder representations of $\Sp_{2n}(F)$ for simplicity. 

\begin{defi}
A \emph{unipotent ladder datum} for $G_n = \Sp_{2n}(F)$
is a triple $\LL = (X,l,\eta)$ satisfying the following conditions. 
\begin{enumerate}
\item
$X = \{|\cdot|^{x_1}, |\cdot|^{x_2}, \dots, |\cdot|^{x_t}\}$
with 
\begin{itemize}
\item
$x_i \in \Z$; 
\item
$x_1<x_2<\dots<x_t$;
\item
$x_1+x_2+\dots+x_t+\half{t-1} = n$ (so that $t$ is odd). 
\end{itemize}

\item
$l \in \Z$ such that 
\begin{itemize}
\item
$0 \leq 2l \leq t$; 
\item
$x_j + x_{t-j+1} \geq 0$ for $1 \leq j \leq l$;
\item
$x_i \geq 0$ for $l+1 \leq i \leq t-l$. 
\end{itemize}

\item
$\eta \in \{\pm1\}$ such that $(-1)^{[\half{t}]+l} \eta^t = 1$.
\end{enumerate}
We define $\pi(\LL) \in \Irr(G_n)$ by the unique irreducible subrepresentation of 
the standard module
\[
\Delta[x_1,-x_t] \times \Delta[x_2,-x_{t-1}] \times \dots \times \Delta[x_l,-x_{t-l+1}] \rtimes \pi(\phi,\ep),
\]
where 
\[
\phi \coloneqq \bigoplus_{i=l+1}^{t-l} S_{2x_{i}+1}
\]
and $\ep(S_{2x_i+1}) \coloneqq (-1)^{i-l-1}\eta$ for $l+1 \leq i \leq t-l$. 
We call $\pi(\LL)$ a unipotent \emph{ladder representation} of $G_n = \Sp_{2n}(F)$.
\end{defi}

For the notation of parabolically induced representations and standard modules of $G_n$, 
see \S \ref{rep} and \S \ref{classification}, respectively.
The class of ladder representations of $G_n$ contains 
\begin{itemize}
\item
strongly positive discrete series
(see, e.g., \cite{Ma1} for this notion); and 
\item
irreducible representations with irreducible $A$-parameters. 
\end{itemize}
\vskip 10pt

As an analogue of \cite[\S 3]{LM}, 
for a unipotent ladder datum $\LL$,
we will define a directed $3$-colorable graph $\EE(\LL)$ in Definition \ref{graph}. 
As in Proposition \ref{der}, this graph knows derivatives of $\pi(\LL)$. 
In particular, one can see that
the class of ladder representations of $G_n$ for all $n \geq 0$ is preserved by derivatives.
More generally, we can compute all Jacquet modules of $\pi(\EE)$
by Theorem \ref{jacquet} together with \cite[Lemma 2.6]{At}.
It would help in the unitary dual problem for $G_n$. 
As another application of the graph $\EE(\LL)$, 
one can compute the Aubert dual of $\pi(\LL)$ by Theorem \ref{aubert}. 
\vskip 10pt

Among irreducible representations of the form $L(\Delta[x_1,y_1], \dots, \Delta[x_t,y_t])$
with $x_1,\dots,x_t \in \Z$, 
the ladder representations of $\GL_n(F)$ are characterized so that all Jacquet modules are completely reduced. 
This fact was proven by Ram \cite[Theorem 4.1 (c)]{R} in the context of affine Hecke algebras, 
and by Gurevich \cite[Corollary 4.11]{G} in general. 
In Corollary \ref{semi}, we will see that all Jacquet modules of $\pi(\LL)$ are completely reduced as well. 
However, the converse in an appropriate sense does not hold 
(see Examples \ref{counter} and \ref{counter2}). 
\vskip 10pt

As a main result of this paper, 
we give an analogue of the determinantal formula in Theorem \ref{determinantal}.
For unipotent ladder representations of $G_n = \Sp_{2n}(F)$, it is stated as follows.

\begin{thm}\label{main}
Let $\LL = (X,l,\eta)$ be a unipotent ladder datum for $G_n = \Sp_{2n}(F)$
with $X = \{|\cdot|^{x_1},\dots,|\cdot|^{x_t}\}$ such that $x_1<\dots<x_t$.
Let $\Sy(\LL)$ be the subset of $\Sy_t$ consisting of $\sigma$ such that 
\begin{itemize}
\item
for $1 \leq i < j \leq l$, we have $\sigma(i) < \sigma(j)$; 
\item
for $l+1 \leq i < j \leq t-l$, we have $\sigma(i) < \sigma(j)$; 
\item
if $x_i \leq -1$, then $\sigma^{-1}(i) \leq l$. 
\end{itemize}
For $\sigma \in \Sy(\LL)$, 
set 
\begin{align*}
J_\sigma^+ &\coloneqq \{1 \leq j \leq l \;|\; \sigma(j) < \sigma(t-j+1) \}, \\
J_\sigma^- &\coloneqq \{1 \leq j \leq l \;|\; \sigma(j) > \sigma(t-j+1) \},
\end{align*}
and define a direct sum of standard modules $I_\sigma(\LL)$ by 
\[
I_\sigma(\LL) \coloneqq 
\left(
\bigtimes_{j \in J_{\sigma}^+} \Delta[x_{\sigma(j)}, -x_{\sigma(t-j+1)}]
\right)
\rtimes 
\left(
\bigoplus_{\delta \colon J_\sigma^- \rightarrow \{\pm1\}} \pi(\phi_\sigma, \ep_{\sigma,\delta})
\right),
\]
where $(\phi_\sigma, \ep_{\sigma,\delta})$ is given by 
\[
\phi_\sigma \coloneqq 
\left(
\bigoplus_{i=l+1}^{t-l} S_{2x_{\sigma(i)}+1}
\right)
\oplus 
\left(
\bigoplus_{j \in J_\sigma^-} S_{2x_{\sigma(j)}+1} \oplus S_{2x_{\sigma(t-j+1)}+1}
\right)
\]
and 
\begin{align*}
\ep_{\sigma,\delta}(S_{2x_{\sigma(i)}+1}) &\coloneqq (-1)^{i-l-1}\eta,\\
\ep_{\sigma,\delta}(S_{2x_{\sigma(j)}+1}) &=
\ep_{\sigma,\delta}(S_{2x_{\sigma(t-j+1)}+1}) \coloneqq \delta(j)
\end{align*}
for $l+1 \leq i \leq t-l$ and $j \in J_\sigma^-$.
Then in the Grothendieck group $\Rr(G_n)$, 
an equality 
\[
[\pi(\LL)] = p_{\s}\left(\sum_{\sigma \in \Sy(\LL)} \sgn(\sigma) [I_\sigma(\LL)]\right)
\]
holds, 
where $p_\s \colon \Rr(G_n) \rightarrow \Rr(G_n)$ is the $\Z$-linear extension of 
\[
[\pi] \mapsto \left\{
\begin{aligned}
&[\pi] \iif \text{$\pi$ has the same cuspidal support as $\pi(\LL)$},\\
&0 \other
\end{aligned}
\right. 
\]
for each irreducible representation $\pi$ of $G_n$.
\end{thm}

Remark that the projection operator $p_\s$ can be computed by using \cite[Theorem 4.2]{At}. 
In particular, all terms in the right hand side of our formula can be write down explicitly. 
\vskip 10pt

In the following example, 
when $\phi = \oplus_{i=1}^r S_{2x_i+1}$ and $\epsilon_i = \ep(S_{2x_i+1})$ for $1 \leq i \leq r$, 
we write 
\[
\pi(\phi,\ep) = \pi(x_1^{\epsilon_1}, \dots, x_r^{\epsilon_r}).
\] 

\begin{ex}\label{ex0}
Let us consider $\LL \coloneqq (\{|\cdot|^0, |\cdot|^1, |\cdot|^2\}, 1, +1)$. 
Then $\pi(\LL)$ is the unique irreducible subrepresentation of 
$\Delta[0,-2] \rtimes \pi(1^+)$.
We give the list of $\sigma \in \Sy(\LL) = \Sy_3$, $\sgn(\sigma)$ and $I_\sigma(\LL)$
in Table \ref{tab0}. 

{\renewcommand\arraystretch{1.3}
\begin{longtable}[c]{ccc}
\caption{List of $\sigma$, $\sgn(\sigma)$, $I_\sigma(\LL)$ in Example \ref{ex0}.}
\label{tab0}
\\
\hline
$\sigma$ & $\sgn(\sigma)$ & $I_\sigma(\LL)$ \\
\hline
$\left(\begin{smallmatrix}
1&2&3\\
1&2&3
\end{smallmatrix}\right)$
& $1$ &
$\Delta[0,-2] \rtimes \pi(1^+)$\\
$\left(\begin{smallmatrix}
1&2&3\\
1&3&2
\end{smallmatrix}\right)$
& $-1$ &
$\Delta[0,-1] \rtimes \pi(2^+)$\\
$\left(\begin{smallmatrix}
1&2&3\\
2&1&3
\end{smallmatrix}\right)$
& $-1$ &
$\Delta[1,-2] \rtimes \pi(0^+)$\\
$\left(\begin{smallmatrix}
1&2&3\\
2&3&1
\end{smallmatrix}\right)$
& $1$ &
$\pi(0^+,1^+,2^+) \oplus \pi(0^-,1^-,2^+)$\\
$\left(\begin{smallmatrix}
1&2&3\\
3&1&2
\end{smallmatrix}\right)$
& $1$ &
$\pi(0^+,1^+,2^+) \oplus \pi(0^+,1^-,2^-)$\\
$\left(\begin{smallmatrix}
1&2&3\\
3&2&1
\end{smallmatrix}\right)$
& $-1$ &
$\pi(0^+,1^+,2^+) \oplus \pi(0^-,1^+,2^-)$
\end{longtable}
}

Since $p_\s([\pi(0^-,1^+,2^-)]) = 0$, Theorem \ref{main} asserts that 
\begin{align*}
[\pi(\LL)] &=
[\Delta[0,-2] \rtimes \pi(1^+)] - [\Delta[0,-1] \rtimes \pi(2^+)] - [\Delta[1,-2] \rtimes \pi(0^+)]
\\&\quad
+[\pi(0^+,1^+,2^+)] + [\pi(0^-,1^-,2^+)] + [\pi(0^+,1^-,2^-)].
\end{align*}
\end{ex}
\vskip 10pt

This paper is organized as follows. 
First of all, we recall some general results on representation theory of $p$-adic classical groups in \S \ref{pre}.
In \S \ref{sec:ladder}, we define ladder representations and establish several properties. 
Finally, in \S \ref{sec:det}, we state and prove our determinantal formula (Theorem \ref{determinantal}).
\vskip 10pt

\subsection*{Acknowledgement}
We would like to thank Alberto M{\'i}nguez
for proposing this project and for useful discussions.
Especially, he gave an idea to use the projection operator $p_\s \colon \Rr(G_n) \rightarrow \Rr(G_n)$
for our determinantal formula.
To answer his questions, we found significant examples 
(Examples \ref{counter} and \ref{counter2}). 
The author was supported by JSPS KAKENHI Grant Number 19K14494.

\section{Preliminaries}\label{pre}
In this section, we introduce our notation. 

\subsection{Notation}
Throughout this paper, 
we fix a non-archimedean local field $F$ of characteristic zero with normalized absolute value $|\cdot|$. 
\par

For the group $G$ of $F$-points of a connected reductive group defined over $F$, 
representations of $G$ mean that smooth representations over $\C$, 
that is, the stabilizer of every vector is an open subgroup of $G$. 
We denote by $\Rep(G)$ the category of smooth representations of $G$ of finite length, 
and by $\Irr(G)$ the set of equivalence classes of irreducible objects of $\Rep(G)$. 
Let $\Rr(G)$ be the Grothendieck group of $\Rep(G)$. 
The canonical map from the objects of $\Rep(G)$ to $\Rr (G)$ will be denoted by $\pi\mapsto[\pi]$.
\par

Fix a minimal $F$-parabolic subgroup $P_0$ of $G$.
In this paper, every parabolic subgroup of $G$ is assumed to be \emph{standard}, 
i.e., it contains $P_0$.
Let $P = MN$ be a (standard) parabolic subgroup of $G$. 
For a representation $\tau$ of $M$, 
we denote by $\Ind_P^G(\tau)$ the normalized parabolically induced representation of $\tau$. 
We view $\Ind_P^G$ as a functor. 
Its left adjoint, the Jacquet functor with respect to $P$, will be denoted by $\Jac_P$.
\par

We say that an irreducible representation $\pi$ of $G$ is \emph{supercuspidal} 
if it is not a composition factor of any representation of the form $\Ind^{G}_P(\tau)$, 
where $P = MN$ is a proper parabolic subgroup of $G$ and $\tau$ is a representation of $M$. 
We  write $\Cusp(G)$ for the subset of  $\Irr(G)$ consisting of supercuspidal representations.
For any $\pi \in \Rep(G)$, we denote by $\pi^\vee$ the contragredient of $\pi$.
\par

\subsection{Representations of general linear groups}\label{rep0}
Set 
\[
\Irr^\GL \coloneqq \bigcup_{n \geq 0}\Irr (\GL_n(F)),
\quad
\Cusp^\GL \coloneqq \bigcup_{n \geq 0} \Cusp(\GL_n(F)).
\]
For $\rho \in \Cusp^\GL$, let $d_\rho \geq 0$ be such that $\rho \in \Cusp(\GL_{d_\rho}(F))$.
We also write 
\[
\Rr^\GL \coloneqq \bigoplus_{n \geq 0} \Rr(\GL_n(F)).
\]
\par

For $\tau_i \in \Rep(\GL_{n_i}(F))$ for $i=1,\dots,r$, 
we denote the normalized parabolically induced representation by 
\[
\tau_1 \times \dots \times \tau_r \coloneqq 
\Ind_{P_{(n_1,\dots,n_r)}}^{\GL_{n_1+\dots+n_r}(F)}(\tau_1 \boxtimes \dots \boxtimes \tau_r), 
\]
where $P_{(n_1,\dots,n_r)}$ is the parabolic subgroup of $\GL_{n_1+\dots+n_r}(F)$
with Levi subgroup $\GL_{n_1}(F) \times \dots \times \GL_{n_r}(F)$. 
We define a map $m \colon \Rr^\GL \times \Rr^\GL \rightarrow \Rr^\GL$ 
by the $\Z$-linear extension of 
\[
[\tau_1] \times [\tau_2] \mapsto [\tau_1 \times \tau_2]
\]
for $\tau_i \in \Irr(\GL_{n_i}(F))$.
It gives a $\Z$-graded ring structure on $\Rr^\GL$. 
\par

The Jacquet functor for $\GL_{n_1+\dots+n_r}(F)$ with respect to $P_{(n_1,\dots,n_r)}$ 
is denoted by $\Jac_{(n_1,\dots,n_r)} = \Jac_{P_{(n_1,\dots,n_r)}}$. 
We define a map $m^* \colon \Rr^\GL \rightarrow \Rr^\GL \otimes \Rr^\GL$
by the $\Z$-linear extension of 
\[
[\tau] \mapsto \sum_{k=0}^n\left[\Jac_{(k,n-k)}(\tau)\right]
\]
for $\tau \in \Irr(\GL_n(F))$.
It gives a co-multiplication of $\Rr^\GL$, i.e., $m^*$ is a ring homomorphism.
\par

Finally, we set
\[
M^* \colon \Rr^\GL \rightarrow \Rr^\GL\otimes \Rr^\GL 
\]
to be the composition $M^* \coloneqq (m\otimes \id) \circ (\cdot^\vee\otimes m^*) \circ s \circ m^*$, 
where $s \colon \Rr^\GL\otimes \Rr^\GL \to \Rr^\GL\otimes \Rr^\GL $ denotes 
the transposition $s(\sum_i \tau_i\otimes \tau'_i) \coloneqq \sum_i \tau'_i\otimes \tau_i$.
\par

\subsection{Representations of classical groups}\label{rep}
In this paper, we let $G_n$ be either the split special orthogonal group $\SO_{2n+1}(F)$ 
or the symplectic group $\Sp_{2n}(F)$ of rank $n$. 
Set 
\[
\Irr^{G} \coloneqq \bigcup_{n \geq 0} \Irr (G_n), 
\quad
\Cusp^G \coloneqq \bigcup_{n \geq 0}\Cusp(G_n)
\]
and 
\[
\Rr^G \coloneqq \bigoplus_{n \geq 0} \Rr(G_{n}), 
\]
where the unions and the direct sum are taken over groups of the same type. 
\par

Let $P$ be the standard parabolic subgroup of $G_n$ with Levi subgroup isomorphic to
$\GL_{d_1}(F) \times \dots \times \GL_{d_r}(F) \times G_{n_0}$. 
For $\pi_0 \in \Rep(G_{n_0})$ and $\tau_i \in \Rep(\GL_{d_i}(F))$ with $1 \leq i \leq r$, 
we denote the normalized parabolically induced representation by
\[
\tau_1 \times \dots \times \tau_r \rtimes \pi_0 
\coloneqq 
\Ind_{P}^{G_n}(\tau_1 \boxtimes \dots \boxtimes \tau_r \boxtimes \pi_0).
\]
It gives an $\Rr^\GL$-module structure on $\Rr^G$ by 
\[
\rtimes \colon \Rr^\GL \otimes \Rr^G \rightarrow \Rr^G,\;
[\tau] \otimes [\pi] \mapsto [\tau \rtimes \pi]
\]
for $\tau \in \Irr(\GL_d(F))$ and $\pi \in \Irr(G_n)$. 
\par

On the other hand, Jacquet functors give a comodule structure 
\[
\mu^* \colon \Rr^G \rightarrow \Rr^\GL \otimes \Rr^G
\]
by the $\Z$-linear extension of 
\[
[\pi] \mapsto \sum_{k=0}^{n} \left[\Jac_{P_k}^{G_n}(\pi)\right]
\]
for $\pi \in \Irr(G_n)$, 
where $P_k$ is the standard parabolic subgroup of $G_n$ 
with Levi subgroup isomorphic to $\GL_{k}(F) \times G_{n-k}$. 
The Geometric Lemma at the level of Grothendieck groups is commonly known in this case 
as \emph{Tadi{\'c}'s formula}.
\begin{prop}[Tadi{\'c}'s formula \cite{T}]\label{Tadic}
For $[\tau] \in \Rr^{\GL}$ and $[\pi] \in \Rr^G$, we have
\[
\mu^*([\tau] \rtimes [\pi])=M^*([\tau]) \rtimes \mu^*([\pi]).
\]
\end{prop}
\par

For $\pi \in \Irr(G_n)$, 
there exist $\rho_1, \dots, \rho_r \in \Cusp^\GL$ and $\sigma \in \Cusp^G$
such that $\pi$ is a subquotient of $\rho_1 \times \dots \times \rho_r \rtimes \sigma$. 
The multi-set 
\[
\supp(\pi) \coloneqq \{\rho_1, \dots, \rho_r, \rho_1^\vee, \dots, \rho_r^\vee\} \cup \{\sigma\} 
\]
is uniquely determined by $\pi$.
We call $\supp(\pi)$ the \emph{cuspidal support} of $\pi$. 
Similarly, when $\rho_1 \times \dots \times \rho_r \rtimes \sigma$ is a representation of $G_n$, 
we say that the multi-set $\{\rho_1, \dots, \rho_r, \rho_1^\vee, \dots, \rho_r^\vee\} \cup \{\sigma\}$
is a \emph{cuspidal support} of $G_n$. 
For a cuspidal support $\s$ of $G_n$, 
we define a map $p_\s \colon \Rr(G_n) \rightarrow \Rr(G_n)$ 
by the $\Z$-linear extension of 
\[
[\pi] \mapsto \left\{
\begin{aligned}
&[\pi] \iif \supp(\pi) = \s, \\
&0 \other
\end{aligned}
\right.
\]
for $\pi \in \Irr(G_n)$. 
By the Langlands classification explained in \S \ref{classification} below, 
the computation of $p_\s([\pi])$ for $\pi \in \Irr(G_n)$
is reduced to the case where $\pi$ is tempered.
Since the cuspidal supports of irreducible tempered representations 
can be computed by \cite[Theorem 4.2]{At}, 
one can easily determine $p_\s([\pi])$ for any $\pi \in \Irr(G_n)$. 
\par

Fix $\rho \in \Cusp(\GL_d(F))$. 
For $\pi \in \Rep(G_n)$ with $n \geq d$, 
we define the \emph{$\rho$-derivative} $D_\rho(\pi)$ of $\pi$ by 
the semisimple representation of $G_{n-d}$ satisfying that
\[
[\Jac_{P_d}(\pi)] = [\rho] \boxtimes [D_\rho(\pi)] + \sum_i [\tau_i] \otimes [\pi_i], 
\]
where $\tau_i \in \Irr(\GL_d(F))$ and $\pi_i \in \Irr(G_{n-d})$ such that $\tau_i \not\cong \rho$. 
When $n < d$, we understand that $D_\rho(\pi) = 0$ for any $\pi \in \Rep(G_n)$.
By the $\Z$-linear extension, we regard $D_\rho$ as a map $\Rr^G \rightarrow \Rr^G$.
We note that
for a cuspidal support $\s$ of $G_n$, 
we have
\[
D_\rho \circ p_\s = 
\left\{
\begin{aligned}
&p_{\s'} \circ D_\rho \iif \rho \in \s, \\
&0 \other, 
\end{aligned}
\right.
\]
where $\s'$ is the cuspidal support of $G_{n-d}$ such that $\s = \s' \cup \{\rho, \rho^\vee\}$.
\par

We say that $\rho \in \Cusp^\GL$ is \emph{good} 
if 
\begin{itemize}
\item
$\rho^\vee \cong \rho|\cdot|^a$ for some $a \in \Z$; and 
\item
$\rho|\cdot|^m \rtimes \1_{G_0}$ is reducible for some $m \in \Z$. 
\end{itemize}
Note that
for any $\rho \in \Cusp^\GL$ satisfying that $\rho \cong \rho^\vee$, 
exactly one of $\rho$ or $\rho|\cdot|^{1/2}$ is good. 
An irreducible representation $\pi$ of $G_n$ is said to be \emph{of good parity}
if every $\rho \in \supp(\pi) \cap \Cusp^\GL$ is good.

\subsection{Langlands classification}\label{classification}
For  $\tau \in \Rep(\GL_n(F))$ and a character $\chi$ of $F^\times$, 
we denote by $\tau\chi$ the representation obtained from $\tau$ 
by twisting by the character $\chi \circ \det$. 
A \emph{segment} $[x,y]_\rho$ is a set of supercuspidal representations of the form
\[
\{\rho|\cdot|^{x} ,\rho|\cdot|^{x-1} ,\dots ,\rho|\cdot|^{y}\},
\]
where $\rho \in \Cusp^\GL$ and $x,y \in \R$ with $x-y \in \Z$ and $x \geq y$.
For a segment $[x,y]_\rho$, 
we define the \emph{Steinberg representation} $\Delta_{\rho}[x,y]$ of $\GL_{d_\rho(x-y+1)}(F)$ by
the unique irreducible subrepresentation of 
\[
\rho|\cdot|^{x} \times \rho|\cdot|^{x-1} \times \dots \times \rho|\cdot|^{y}. 
\]
It is an essentially discrete series 
and all essentially discrete series are of this form \cite[Theorem 9.3]{Z}. 
By convention, we set $\Delta_\rho[x,x+1]$ to be the trivial representation of the trivial group $\GL_0(F)$,
and $\Delta_\rho[x,y] \coloneqq 0$ for $y>x+1$.
\par

Let $[x_1,y_1]_{\rho_1}, \dots, [x_r,y_r]_{\rho_r}$ be segments, 
where $\rho_i \in \Cusp^\GL$ is assumed to be unitary for $1 \leq i \leq r$, 
and let $\pi_\temp$ be an irreducible tempered representation of $G_{n_0}$. 
When $x_1+y_1 \leq \dots \leq x_r+y_r < 0$, 
the parabolically induced representation 
\[
\Delta_{\rho_1}[x_1,y_1] \times \dots \times \Delta_{\rho_r}[x_r,y_r] \rtimes \pi_\temp
\]
is called a \emph{standard module}. 
It has a unique irreducible subrepresentation, denoted by 
\[
L(\Delta_{\rho_1}[x_1,y_1], \dots, \Delta_{\rho_r}[x_r,y_r]; \pi_\temp).
\]
Conversely, the Langlands classification says that 
any irreducible representation $\pi$ of $G_n$
is of this form for some standard module.
It is known that 
the isomorphism classes of standard modules of $G_n$ give a basis of $\Rr(G_n)$. 
\par

By the local Langlands correspondence established by Arthur \cite{Ar}, 
every irreducible tempered representation $\pi_\temp$ of $G_n$
is written as $\pi_\temp = \pi(\phi,\ep)$, 
where $\phi$ is a \emph{tempered $L$-parameter} for $G_n$ 
and $\ep$ is a character of the component group of $\phi$. 
When $\pi_\temp = \pi(\phi,\ep)$ is of good parity, they are described as follows: 
\begin{itemize}
\item
$\phi$ is decomposed into a (formal) sum 
\[
\phi = \bigoplus_{i=1}^r \rho_i \boxtimes S_{a_i}, 
\]
where $\rho_i \in \Cusp^\GL$ is assumed to be unitary, and 
$S_{a_i}$ is a unique irreducible algebraic representation of $\SL_2(\C)$ of dimension $a_i$
such that 
\begin{enumerate}
\item
$\rho_i|\cdot|^{\half{a_i-1}}$ is good for $i=1,\dots,r$; 
\item
$\sum_{i=1}^r d_{\rho_i}a_i = 2n$ if $G_n = \SO_{2n+1}(F)$
(\resp $\sum_{i=1}^r d_{\rho_i}a_i = 2n+1$ if $G_n = \Sp_{2n}(F)$); 
\item
$\prod_{i=1}^r \omega_{i}^{a_i} = \1$, where $\omega_i$ is the central character of $\rho_i$;
\end{enumerate}

\item
$\ep$ is characterized by a tuple of sign 
$(\ep(\rho_1 \boxtimes S_{a_1}), \dots, \ep(\rho_r \boxtimes S_{a_r})) \in \{\pm1\}^r$
such that 
\begin{enumerate}
\setcounter{enumi}{3}
\item
if $\rho_i \cong \rho_j$ and $a_i = a_j$, 
then $\ep(\rho_i \boxtimes S_{a_i}) = \ep(\rho_j \boxtimes S_{a_j})$; 
\item
$\prod_{i=1}^r \ep(\rho_i \boxtimes S_{a_i}) = 1$.
\end{enumerate}
\end{itemize}
By convention, let $S_0$ be the zero representation of $\SL_2(\C)$.
When $\rho|\cdot|^{\half{1}}$ is good, 
we formally allow the situation $\phi \supset \rho \boxtimes S_0$. 
In this case, we set
\[
\pi(\phi,\ep) \coloneqq 
\left\{
\begin{aligned}
&\pi(\phi-\rho \boxtimes S_0,\ep) \iif \ep(\rho \boxtimes S_0) = 1, \\
&0 \iif \ep(\rho \boxtimes S_0) = -1.
\end{aligned}
\right.
\]

\section{Ladder representations}\label{sec:ladder}
In this section, we define \emph{ladder representations} for $G_n$, 
and establish several properties. 

\subsection{Definition}
Recall that ladder representations of $\GL_n(F)$
are obtained by appropriate Jacquet modules of \emph{Speh representations}.
Inspired by M{\oe}glin's construction of irreducible representations with $A$-paramters, 
we introduce the following notion.

\begin{defi}\label{def:ladder}
We say that $\LL = (\ub{X}, \ub{l}, \ub{\eta})$ is a \emph{ladder datum} for $G_n$
if the following conditions hold.

\begin{enumerate}
\item
$\ub{X} = (X_\rho)_\rho$, $\ub{l} = (l_\rho)_\rho$ and $\ub{\eta} = (\eta_\rho)_\rho$, 
where $\rho$ runs over $\Cusp^\GL$ such that $\rho^\vee \cong \rho$.

\item
$X_\rho$ is a finite subset of $\Cusp^\GL$ of the form 
\[
X_\rho = \{ \rho|\cdot|^{x_1}, \dots, \rho|\cdot|^{x_{t_\rho}}\}
\]
such that 
\begin{itemize}
\item
all but finitely many $\rho$, the cardinality $t_\rho$ is zero; 
\item
$\rho|\cdot|^{x_j}$ is good (so that $2x_i \in \Z$) for $1 \leq j \leq t_\rho$; 
\item
$x_1 < \dots < x_{t_\rho}$, which depend on $\rho$;
\item
the sum
\[
\sum_{\rho} (2x_1+\dots+2x_{t_\rho}+t_\rho)d_\rho
\]
is equal to $2n$ if $G_n = \SO_{2n+1}(F)$ (\resp $2n+1$ if $G_n = \Sp_{2n}(F)$).
\end{itemize}

\item
$l_\rho \in \Z$ such that 
\begin{itemize}
\item
$0 \leq 2l_\rho \leq t_\rho$; 
\item
$x_j+x_{t_\rho-j+1} \geq 0$ for $1 \leq j \leq l_\rho$ (especially $x_{t_\rho-l_\rho+1} > 0$);
\item
$x_i \geq -1/2$ for $l_\rho+1 \leq i \leq t_\rho-l_\rho$.
\end{itemize}

\item
$\eta_\rho \in \{\pm1\}$ such that 
\begin{itemize}
\item
$\eta_\rho = 1$ if $X_\rho = \emptyset$ or if $x_{l_\rho+1} = -1/2$;
\item
$\eta_\rho = -1$ if $2l_\rho = t_\rho$;
\item
the equation
\[
\prod_{\rho}(-1)^{[\half{t_\rho}]+l_\rho} \eta_\rho^{t_\rho} = 1
\]
holds.
\end{itemize}
\end{enumerate}
Define $\pi(\LL) \in \Irr(G_n)$ by 
the unique irreducible subrepresentation of the standard module
\[
I(\LL) \coloneqq
\bigtimes_{\rho} 
\left(\Delta_\rho[x_1, -x_{t_\rho}] \times \Delta_\rho[x_2,-x_{t_\rho-1}] \times \dots \times 
\Delta_\rho[x_{l_\rho}, -x_{t_\rho-l_\rho+1}]
\right)
\rtimes \pi(\phi, \ep), 
\]
where 
\[
\phi \coloneqq \bigoplus_{\rho} \bigoplus_{i=l_\rho+1}^{t_\rho-l_\rho} \rho \boxtimes S_{2x_i+1}
\]
and 
$\ep(\rho \boxtimes S_{2x_i+1}) \coloneqq (-1)^{i-l_\rho-1}\eta_\rho$ 
for $l_\rho+1 \leq i \leq t_\rho-l_\rho$.
We call $\pi(\LL)$ a \emph{ladder representation} of $G_n$.
\end{defi}

Note that $\pi(\LL)$ is an irreducible representation of good parity of $G_n$.
Moreover, the map $\LL \mapsto \pi(\LL)$ is injective.

\begin{rem}
Let $\LL = (\ub{X}, \ub{l}, \ub{\eta})$ be a ladder datum 
with $\ub{X} = (X_\rho)_\rho$ and $\ub{l} = (l_\rho)_\rho$.
\begin{enumerate}
\item
If $l_\rho = 0$ for all $\rho$, 
then $\pi(\LL)$ is \emph{strongly positive discrete series}. 
For this notion, see \cite{Ma1} for example. 
Conversely, every strongly positive discrete series representation 
can be written as $\pi(\LL)$ for some ladder datum $\LL = (\ub{X}, \ub{l}, \ub{\eta})$ with $l_\rho = 0$ for all $\rho$. 

\item
If $X_\rho$ is a segment $[x,y]_\rho$ for all $\rho$, 
then $\pi(\LL)$ is an irreducible representation with an $A$-parameter. 
In this case, the definition of $\pi(\LL)$ is exactly the same as M{\oe}glin's construction. 
For more precision, see \cite{X2}.
Conversely, if $\pi$ is an irreducible representation with an \emph{irreducible} $A$-parameter, 
then $\pi$ can be written as $\pi(\LL)$ for some ladder datum $\LL = (\ub{X}, \ub{l}, \ub{\eta})$ 
such that $X_\rho$ is a segment for some $\rho$ and $X_{\rho'} = \emptyset$ for $\rho' \not\cong \rho$. 

\end{enumerate}
\end{rem}

\subsection{The graph}
Fix $\rho \in \Cusp^\GL$ such that $\rho^\vee \cong \rho$. 
For a ladder datum $\LL = (\ub{X},\ub{l},\ub{\eta})$, 
we write $\LL_\rho = (X_\rho,l_\rho,\eta_\rho)$. 
As an analogue of \cite[\S 3]{LM}, we attach a directed $3$-colorable graph to $\LL_\rho$. 

\begin{defi}\label{graph}
Write $X_\rho = \{\rho|\cdot|^{x_1},\dots,\rho|\cdot|^{x_t}\}$ with $x_1 < \dots < x_t$ 
and $(l,\eta) = (l_\rho, \eta_\rho)$.
We define a directed $3$-colorable graph 
$\EE(\LL_\rho) = (\VV(\LL_\rho), \less, f_\rho)$ as follows.

\begin{enumerate}
\item
The vertex set of $\EE(\LL_\rho)$ is
\[
\VV(\LL_\rho) \coloneqq 
\{
(x_{l+1+i},-i), (x_{l+1+i}-1,-i), \dots, (-x_{t-l-i},-i) \;|\; -l \leq i \leq t-l-1
\}.
\]

\item
The edges of $\EE(\LL_\rho)$ consist of
the horizontal arrows and the diagonal arrows 
assigned as follows:
\begin{itemize}
\item
If $x_1,\dots,x_t \in \Z$, 
then 
the horizontal arrows are 
\[
(j,-i) \less (j-1,-i)
\]
for $-l \leq i \leq t-l-1$ and $-x_{t-l-i}+1 \leq j \leq x_{l+1+i}$, 
and the diagonal arrows are
\[
(j,-i) \less (j+1, -(i+1))
\]
for $-l \leq i < t-l-1$ and $-x_{t-l-(i+1)}-1 \leq j \leq x_{l+1+i}$. 

\item
If $x_1,\dots,x_t \in (1/2)\Z\setminus\Z$, 
then
the horizontal arrows are 
\[
(j+1/2,-i) \less (j-1/2,-i)
\]
for $-l \leq i \leq t-l-1$ and $-x_{t-l-i}+1 \leq j+1/2 \leq x_{l+1+i}$, 
and the diagonal arrows are
\[
(j-1/2,-i) \less (j+1/2, -(i+1))
\]
for $-l \leq i < t-l-1$ and $-x_{t-l-(i+1)}-1 \leq j-1/2 \leq x_{l+1+i}$. 
\end{itemize}

\item
The coloring map $f_\rho \colon \VV(\LL_\rho) \rightarrow \{-1,0,1\}$ is given as follows:
\begin{itemize}
\item
If $x_1,\dots,x_t \in \Z$, then 
\[
f_\rho(j,-i) \coloneqq
\left\{
\begin{aligned}
&(-1)^{j}\eta \iif 0 \leq i \leq t-2l-1,\; i-(t-2l-1) \leq j \leq i, \\
&0 \other.
\end{aligned}
\right.
\]
Here, we note that $x_{l+1+i} \geq i$ for $0 \leq i \leq t-2l-1$.

\item
If $x_1,\dots,x_t \in (1/2)\Z\setminus\Z$, 
then
\[
f_\rho\left(j-\half{1}\eta,-i\right) \coloneqq
\left\{
\begin{aligned}
&(-1)^{j}\eta \iif 0 \leq i \leq t-2l-1,\; i-(t-2l-1)+\eta \leq j \leq i, \\
&0 \other.
\end{aligned}
\right.
\]
Here, we note that $x_{l+1+i} \geq i-\half{1}\eta$ for $0 \leq i \leq t-2l-1$.
\end{itemize}

\end{enumerate}
\end{defi}

\begin{rem}
\begin{enumerate}
\item
The map $\LL_\rho \mapsto \EE(\LL_\rho) = (\VV(\LL_\rho), \less, f_\rho)$ is injective. 
In particular, one can specify $\EE(\LL_\rho)$ to assign $\LL_\rho$.

\item
Note that $f_\rho(x,y) = 0 \iff f_\rho(-x,-(t-2l-1)-y) = 0$. 
Moreover, if $(x,y) = (-x,-(t-2l-1)-y)$, 
then $t$ is odd and $(x,y) = (0,-(\half{t-2l-1}))$ so that $f_\rho(x,y) \not= 0$. 
Hence $|f_\rho^{-1}(0)|$ is even.
We will write $2m_\rho = |f_\rho^{-1}(0)|$.

\item
For each $x \in \R$, 
the graph $\EE(\LL_\rho)$ has at most one minimal vertex of the form $(x,y)$ for $y \in \R$.
\end{enumerate}
\end{rem}

\begin{ex}
\begin{enumerate}
\item
Consider $\LL_\rho = (\{\rho|\cdot|^{1/2},\rho|\cdot|^{5/2},\rho|\cdot|^{7/2}\}, 0, -1)$.
Then the graph $\EE(\LL_\rho)$ is as follows:
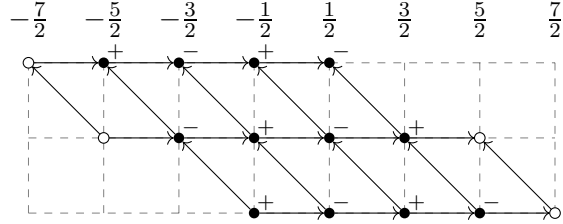
\begin{center}
\begin{tikzpicture}
\draw[help lines, step=1,dashed] (0,0) grid (7,2);
\node at (0,2.2) [above] {$-\half{7}$};
\node at (1,2.2) [above] {$-\half{5}$};
\node at (2,2.2) [above] {$-\half{3}$};
\node at (3,2.2) [above] {$-\half{1}$};
\node at (4,2.2) [above] {$\half{1}$};
\node at (5,2.2) [above] {$\half{3}$};
\node at (6,2.2) [above] {$\half{5}$};
\node at (7,2.2) [above] {$\half{7}$};

\filldraw [fill = white] 
(0,2) circle (2pt);
\fill 
(4,2) circle (2pt) 
(3,2) circle (2pt) 
(2,2) circle (2pt) 
(1,2) circle (2pt) ; 
\node at (4,2) [above right, inner sep=1] {\scriptsize $-$};
\node at (3,2) [above right, inner sep=1] {\scriptsize $+$};
\node at (2,2) [above right, inner sep=1] {\scriptsize $-$};
\node at (1,2) [above right, inner sep=1] {\scriptsize $+$};
\draw[->](3+0.07,2)--(4-0.07,2);
\draw[->](2+0.07,2)--(3-0.07,2);
\draw[->](1+0.07,2)--(2-0.07,2);
\draw[->](0+0.07,2)--(1-0.07,2);

\filldraw [fill = white] 
(6,1) circle (2pt) 
(1,1) circle (2pt);
\fill 
(5,1) circle (2pt) 
(4,1) circle (2pt) 
(3,1) circle (2pt) 
(2,1) circle (2pt) ; 
\node at (5,1) [above right, inner sep=1] {\scriptsize $+$};
\node at (4,1) [above right, inner sep=1] {\scriptsize $-$};
\node at (3,1) [above right, inner sep=1] {\scriptsize $+$};
\node at (2,1) [above right, inner sep=1] {\scriptsize $-$};
\draw[->](5+0.07,1)--(6-0.07,1);
\draw[->](4+0.07,1)--(5-0.07,1);
\draw[->](3+0.07,1)--(4-0.07,1);
\draw[->](2+0.07,1)--(3-0.07,1);
\draw[->](1+0.07,1)--(2-0.07,1);

\filldraw [fill = white] 
(7,0) circle (2pt) ;
\fill 
(6,0) circle (2pt) 
(5,0) circle (2pt) 
(4,0) circle (2pt) 
(3,0) circle (2pt) ;
\node at (6,0) [above right, inner sep=1] {\scriptsize $-$};
\node at (5,0) [above right, inner sep=1] {\scriptsize $+$};
\node at (4,0) [above right, inner sep=1] {\scriptsize $-$};
\node at (3,0) [above right, inner sep=1] {\scriptsize $+$};
\draw[->](6+0.07,0)--(7-0.07,0);
\draw[->](5+0.07,0)--(6-0.07,0);
\draw[->](4+0.07,0)--(5-0.07,0);
\draw[->](3+0.07,0)--(4-0.07,0);

\draw[->](5-0.05,1+0.05)--(4+0.05,2-0.05);
\draw[->](4-0.05,1+0.05)--(3+0.05,2-0.05);
\draw[->](3-0.05,1+0.05)--(2+0.05,2-0.05);
\draw[->](2-0.05,1+0.05)--(1+0.05,2-0.05);
\draw[->](1-0.05,1+0.05)--(0+0.05,2-0.05);

\draw[->](7-0.05,0+0.05)--(6+0.05,1-0.05);
\draw[->](6-0.05,0+0.05)--(5+0.05,1-0.05);
\draw[->](5-0.05,0+0.05)--(4+0.05,1-0.05);
\draw[->](4-0.05,0+0.05)--(3+0.05,1-0.05);
\draw[->](3-0.05,0+0.05)--(2+0.05,1-0.05);

\end{tikzpicture}
\end{center}
Here, for a vertex $(x,y) \in \VV(\LL_\rho)$, 
we put a white circle if $f_\rho(x,y) = 0$, 
whereas a block circle if $f_\rho(x,y) \not= 0$.

\item
Consider $\LL_\rho = ([4,0]_\rho, 1, -1)$.
Then the graph $\EE(\LL_\rho)$ is as follows:
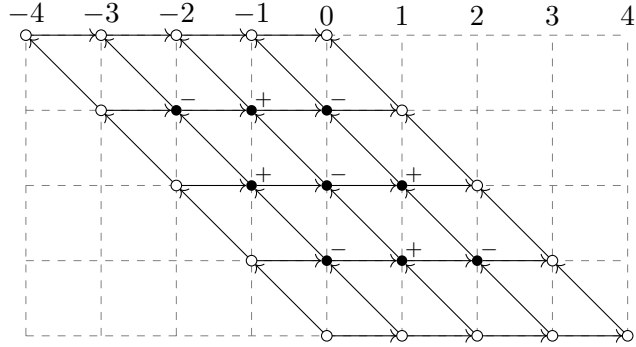
\begin{center}
\begin{tikzpicture}
\draw[help lines, step=1,dashed] (0,0) grid (8,4);
\node at (0,4) [above] {$-4$};
\node at (1,4) [above] {$-3$};
\node at (2,4) [above] {$-2$};
\node at (3,4) [above] {$-1$};
\node at (4,4) [above] {$0$};
\node at (5,4) [above] {$1$};
\node at (6,4) [above] {$2$};
\node at (7,4) [above] {$3$};
\node at (8,4) [above] {$4$};

\filldraw [fill = white] 
(4,4) circle (2pt) 
(3,4) circle (2pt) 
(2,4) circle (2pt) 
(1,4) circle (2pt) 
(0,4) circle (2pt) ;
\draw[->](3+0.07,4)--(4-0.07,4);
\draw[->](2+0.07,4)--(3-0.07,4);
\draw[->](1+0.07,4)--(2-0.07,4);
\draw[->](0+0.07,4)--(1-0.07,4);

\filldraw [fill = white] 
(5,3) circle (2pt) 
(1,3) circle (2pt) ;
\fill 
(4,3) circle (2pt) 
(3,3) circle (2pt) 
(2,3) circle (2pt) ;
\node at (4,3) [above right, inner sep=1] {\scriptsize $-$};
\node at (3,3) [above right, inner sep=1] {\scriptsize $+$};
\node at (2,3) [above right, inner sep=1] {\scriptsize $-$};
\draw[->](4+0.07,3)--(5-0.07,3);
\draw[->](3+0.07,3)--(4-0.07,3);
\draw[->](2+0.07,3)--(3-0.07,3);
\draw[->](1+0.07,3)--(2-0.07,3);

\filldraw [fill = white] 
(6,2) circle (2pt) 
(2,2) circle (2pt) ;
\fill 
(5,2) circle (2pt) 
(4,2) circle (2pt) 
(3,2) circle (2pt) ;
\node at (5,2) [above right, inner sep=1] {\scriptsize $+$};
\node at (4,2) [above right, inner sep=1] {\scriptsize $-$};
\node at (3,2) [above right, inner sep=1] {\scriptsize $+$};
\draw[->](5+0.07,2)--(6-0.07,2);
\draw[->](4+0.07,2)--(5-0.07,2);
\draw[->](3+0.07,2)--(4-0.07,2);
\draw[->](2+0.07,2)--(3-0.07,2);

\filldraw [fill = white] 
(7,1) circle (2pt) 
(3,1) circle (2pt) ;
\fill 
(6,1) circle (2pt) 
(5,1) circle (2pt) 
(4,1) circle (2pt) ;
\node at (6,1) [above right, inner sep=1] {\scriptsize $-$};
\node at (5,1) [above right, inner sep=1] {\scriptsize $+$};
\node at (4,1) [above right, inner sep=1] {\scriptsize $-$};
\draw[->](6+0.07,1)--(7-0.07,1);
\draw[->](5+0.07,1)--(6-0.07,1);
\draw[->](4+0.07,1)--(5-0.07,1);
\draw[->](3+0.07,1)--(4-0.07,1);

\filldraw [fill = white] 
(8,0) circle (2pt) 
(7,0) circle (2pt) 
(6,0) circle (2pt) 
(5,0) circle (2pt) 
(4,0) circle (2pt) ;
\draw[->](7+0.07,0)--(8-0.07,0);
\draw[->](6+0.07,0)--(7-0.07,0);
\draw[->](5+0.07,0)--(6-0.07,0);
\draw[->](4+0.07,0)--(5-0.07,0);

\draw[->](5-0.05,3+0.05)--(4+0.05,4-0.05);
\draw[->](4-0.05,3+0.05)--(3+0.05,4-0.05);
\draw[->](3-0.05,3+0.05)--(2+0.05,4-0.05);
\draw[->](2-0.05,3+0.05)--(1+0.05,4-0.05);
\draw[->](1-0.05,3+0.05)--(0+0.05,4-0.05);

\draw[->](6-0.05,2+0.05)--(5+0.05,3-0.05);
\draw[->](5-0.05,2+0.05)--(4+0.05,3-0.05);
\draw[->](4-0.05,2+0.05)--(3+0.05,3-0.05);
\draw[->](3-0.05,2+0.05)--(2+0.05,3-0.05);
\draw[->](2-0.05,2+0.05)--(1+0.05,3-0.05);

\draw[->](7-0.05,1+0.05)--(6+0.05,2-0.05);
\draw[->](6-0.05,1+0.05)--(5+0.05,2-0.05);
\draw[->](5-0.05,1+0.05)--(4+0.05,2-0.05);
\draw[->](4-0.05,1+0.05)--(3+0.05,2-0.05);
\draw[->](3-0.05,1+0.05)--(2+0.05,2-0.05);

\draw[->](8-0.05,0+0.05)--(7+0.05,1-0.05);
\draw[->](7-0.05,0+0.05)--(6+0.05,1-0.05);
\draw[->](6-0.05,0+0.05)--(5+0.05,1-0.05);
\draw[->](5-0.05,0+0.05)--(4+0.05,1-0.05);
\draw[->](4-0.05,0+0.05)--(3+0.05,1-0.05);
\end{tikzpicture}
\end{center}
\end{enumerate}
\end{ex}

\subsection{Jacquet modules}
The graph $\EE(\LL_\rho) = (\VV(\LL_\rho), \less, f_\rho)$ defined in the previous subsection 
knows Jacquet modules of $\pi(\LL)$. 
First, we consider derivatives of $\pi(\LL)$.
\begin{prop}\label{der}
For $x \in \R$, 
the derivative $D_{\rho|\cdot|^x}(\pi(\LL))$ is nonzero
if and only if 
$\EE(\LL_\rho)$ has a minimal vertex of the form $(x,y)$ for some $y$ 
such that $f_\rho(x,y) = 0$. 
In this case, 
$D_{\rho|\cdot|^x}(\pi(\LL))$ is also a ladder representation.
If we write $D_{\rho|\cdot|^x}(\pi(\LL)) = \pi(\LL')$, 
then $\LL'_{\rho'} = \LL_{\rho'}$ for all $\rho' \not\cong \rho$, 
and the graph $\EE(\LL_\rho')$ is given from $\EE(\LL_\rho)$ 
by removing two vertices $(x,y)$ and $(-x,-(t_\rho-2l_\rho-1)-y)$ from $\VV(\LL_\rho)$, 
where $t_\rho = |X_\rho|$.
\end{prop}
\begin{proof}
Note that $D_{\rho|\cdot|^x} \circ D_{\rho|\cdot|^x}(\pi(\LL)) = 0$ for any $x \in \R$
by Tadi\'c's formula (Proposition \ref{Tadic}) and \cite[Lemma 7.3]{X1}. 
Now the proposition follows from \cite[Theorem 7.1]{AM}.
\end{proof}

As a consequence,
the class of ladder representations is preserved by derivatives, 
which is the same as in the general linear groups case (see \cite{KL}).
Proposition \ref{der} immediately implies the following corollary. 

\begin{cor}\label{sc}
Let $\LL$ be a ladder datum for $G_n$. 
\begin{enumerate}
\item
The ladder representation $\pi(\LL)$ is supercuspidal if and only if $m_\rho = 0$ 
for all $\rho \in \Cusp^\GL$ with $\rho^\vee \cong \rho$. 

\item
More generally, the unique representation $\sigma \in \supp(\pi(\LL)) \cap \Cusp^G$
is a ladder representation with the ladder datum $\LL_\sigma$ such that 
for any $\rho \in \Cusp^\GL$ with $\rho^\vee \cong \rho$, 
the graph $\EE(\LL_{\sigma,\rho})$ is given from $\EE(\LL_\rho)$ 
by removing all vertices $(x,y)$ such that $f_\rho(x,y) = 0$.
\end{enumerate}
\end{cor}

Recall that we have a $\Z$-linear map 
$\mu^* \colon \Rr^G \rightarrow \Rr^\GL \otimes \Rr^G$ in \S \ref{rep}.
For $\pi \in \Irr^G$, 
when 
\[
\mu^*([\pi]) = \sum_{i \in I} [\tau_i] \otimes [\pi_i]
\]
with $\tau_i \in \Irr^\GL$ and $\pi_i \in \Irr^G$, 
for fixed $\rho \in \Cusp^\GL$ with $\rho \cong \rho^\vee$, 
set 
\[
\mu_\rho^*([\pi]) \coloneqq \sum_{i \in I_\rho} [\tau_i] \otimes [\pi_i], 
\]
where
$I_\rho$ is the subset of $I$ consisting of indices $i$ such that 
$\tau_i$ is a subquotient of $\rho|\cdot|^{x_1} \times \dots \times \rho|\cdot|^{x_r}$ for some $x_1,\dots,x_r \in \R$.
By the same argument as \cite[Lemma 2.6]{At}, 
when $\pi \in \Irr^G$ is of good parity, 
$\mu^*([\pi])$ is recovered from $\mu^*_\rho([\pi])$ for all $\rho \in \Cusp^\GL$ with $\rho \cong \rho^\vee$. 
\par

When $\pi$ is a ladder representation of $G_n$, 
one can compute $\mu^*_\rho([\pi])$.
\begin{thm}\label{jacquet}
Let $\LL = (\ub{X}, \ub{l}, \ub{\eta})$ be a ladder datum for $G_n$. 
Write $X_\rho = \{\rho|\cdot|^{x_1}, \dots, \rho|\cdot|^{x_t}\}$ with $x_1 < \dots < x_t$ 
and $(l,\eta) \coloneqq (l_\rho,\eta_\rho)$.
Then 
\[
\mu_\rho^*([\pi(\LL)]) = \sum_{\ub{y} = (y_1,\dots,y_t)} 
\Big[ L(\Delta_\rho[x_1,y_1+1],\dots,\Delta_\rho[x_t,y_t+1]) \Big] \otimes \Big[\pi(\LL_{\ub{y}})\Big], 
\]
where 
$\ub{y} = (y_1,\dots,y_t)$ runs over the subset of $\R^t$ such that 
\begin{itemize}
\item
$-x_{t-i+1}-1 \leq y_i \leq x_i$ and $y_i \equiv x_i \bmod \Z$ for $1 \leq i \leq t$; 
\item
$y_1 < \dots < y_t$; 
\item
$y_i+y_{t-i+1} \geq -1$ for $1 \leq i \leq l$; 
\item
if $l+1 \leq i \leq t-l$, then 
\[
y_i \geq \left\{
\begin{aligned}
&i-l-1 \iif x_1,\dots,x_t \in \Z, \\
&i-l-1-\half{1}\eta \iif x_1,\dots,x_t \in \half{1}\Z \setminus \Z,
\end{aligned}
\right. 
\]
\end{itemize}
and $\LL_{\ub{y}} = (\ub{Y}, \ub{l'}, \ub{\eta})$ is defined by 
\begin{itemize}
\item
$Y_{\rho'} \coloneqq X_{\rho'}$ and $l'_{\rho'} \coloneqq l_{\rho'}$ for $\rho' \not\cong \rho$; 
\item
$Y_\rho \coloneqq \{\rho|\cdot|^{y_i} \;|\; 1 \leq i \leq t,\; y_i+y_{t-i+1} \geq 0\}$; 
\item
$l'_\rho \coloneqq l_\rho - \#\{1\leq i \leq l \;|\; y_i+y_{t-i+1} =-1 \}$.
\end{itemize}
\end{thm}
\begin{proof}
Write $d \coloneqq d_\rho$. 
Let $P_{dk}$ be the standard parabolic subgroup of $G_n$ 
with Levi subgroup isomorphic to $\GL_{dk}(F) \times G_{n-dk}$.
Write 
\[
[\Jac_{P_{dk}}(\pi(\LL))] = \sum_{i \in I} [\tau_i] \otimes [\pi_i]
\]
with $\tau_i \in \Irr(\GL_{dk}(F))$ and $\pi_i \in \Irr(G_{n-dk})$. 
Set $I_\rho$ to be the subset of $I$ consisting of indices $i$ such that 
$\tau_i$ is a subquotient of $\rho|\cdot|^{x_1} \times \dots \times \rho|\cdot|^{x_k}$ for some $x_1,\dots,x_k \in \R$.
Note that by Proposition \ref{der}, 
we see that $\pi_i$ is a ladder representation of $G_{n-dk}$. 
\par

We claim that for a fixed ladder datum $\LL_0$ of $G_{n-dk}$
such that $\LL_{0,\rho'} = \LL_{\rho'}$ for $\rho' \not\cong \rho$,  
the equation 
\[
\sum_{\substack{i \in I_\rho \\ \pi_i \cong \pi(\LL_0)}} [\tau_i]
=
\sum_{\substack{\ub{y} = (y_1,\dots,y_t) \\ \pi(\LL_{\ub{y}}) \cong \pi(\LL_0)}} 
\Big[ L(\Delta_\rho[x_1,y_1+1],\dots,\Delta_\rho[x_t,y_t+1]) \Big]
\]
holds. 
For simplicity, we write $A$ (\resp $B$) for the left (\resp the right) hand side of this equation.
Let $Q$ be the standard parabolic subgroup of $\GL_{dk}(F)$ 
with Levi subgroup isomorphic to $\GL_d(F) \times \dots \times \GL_d(F)$. 
To prove $A=B$, it is enough to show that $[\Jac_Q(A)] = [\Jac_Q(B)]$.
\par

By applying Proposition \ref{der} repeatedly, 
we see that 
\[
[\Jac_Q(A)] = \sum_{\ub{\lambda} = (\lambda_1,\dots,\lambda_k)}
\rho|\cdot|^{\lambda_1} \otimes \dots \otimes \rho|\cdot|^{\lambda_k},
\] 
where $\ub{\lambda} = (\lambda_1,\dots,\lambda_k)$ runs over the subset of $\R^k$ such that 
\begin{itemize}
\item
the graph $\EE_0 \coloneqq \EE(\LL_\rho)$ has a minimal vertex of the form $(\lambda_1,\mu_1)$ for some $\mu_1$
such that $f_\rho(\lambda_1,\mu_1) = 0$; 

\item
for $1 \leq i \leq k-1$, the graph
\[
\EE_{i} \coloneqq \EE_{i-1} \setminus \{(\lambda_i,\mu_i), (-\lambda_{i}, -(t-2l-1)-\mu_i)\}
\]
has a minimal vertex of the form $(\lambda_{i+1},\mu_{i+1})$ for some $\mu_{i+1}$
such that $f_\rho(\lambda_{i+1},\mu_{i+1}) = 0$; 

\item
the graph
\[
\EE_{k} \coloneqq \EE_{k-1} \setminus \{(\lambda_k,\mu_k), (-\lambda_{k}, -(t-2l-1)-\mu_k)\}
\]
is equal to $\EE(\LL_{0,\rho})$.
\end{itemize}
By \cite[Theorem 7 (2)]{LM}, one can compute $[\Jac_Q(B)]$ similarly. 
Then we see that $[\Jac_Q(A)] = [\Jac_Q(B)]$, as desired. 
This completes the proof.
\end{proof}

Note that even if $\ub{y} \not= \ub{y'}$, 
one might have $\LL_{\ub{y}} = \LL_{\ub{y'}}$. 
\begin{ex}
Set $\rho = \1_{\GL_1(F)}$. 
Consider $\LL = (\ub{X}, \ub{l}, \ub{\eta})$ with $X_{\rho'} = \emptyset$ for $\rho' \not\cong \rho$
and 
\[
\LL_{\rho} = (\{\rho|\cdot|^0, \rho|\cdot|^1, \rho|\cdot|^2\}, 1, +1). 
\]
Then $\pi(\LL) = L(\Delta_{\rho}[0,-2]; \pi(\rho \boxtimes S_3,+)) \in \Irr(\Sp_8(F))$. 
If $P$ is the standard Siegel parabolic subgroup, 
then Theorem \ref{jacquet} says that
\[
\Big[\Jac_P(\pi(\LL))\Big] = 
\Big[L(\Delta_\rho[0,-2], \rho|\cdot|^1)\Big] \otimes \Big[\1_{\Sp_0(F)}\Big]
+ 
\Big[L(\Delta_\rho[0,-1], \rho|\cdot|^1, \rho|\cdot|^2)\Big] \otimes \Big[\1_{\Sp_0(F)}\Big].
\]
\end{ex}

By Theorem \ref{jacquet}, we see that
all irreducible subquotients of $\Jac_P(\pi(\LL))$ have distinct cuspidal supports. 
Hence we have: 
\begin{cor}\label{semi}
Let $\pi(\LL)$ be a ladder representation of $G_n$. 
Then for any parabolic subgroup $P$ of $G_n$, 
the Jacquet module $\Jac_P(\pi(\LL))$ is completely reducible. 
\end{cor}

As an analogue of the general groups case \cite[Corollary 4.11]{G}, 
one might expect that the converse of Corollary \ref{semi} would hold in an appropriate sense. 
However, we have a counterexample as follows. 
\begin{ex}\label{counter}
Consider $\pi = \pi(\phi,\ep) \in \Irr(\SO_7(F))$ with 
$\phi = S_2 \oplus S_4$ and $\ep(S_2) = \ep(S_4) = -1$. 
It is a discrete series representation but not ladder.
For $k \in \{1,2,3\}$, 
let $P_k$ be the standard parabolic parabolic subgroup of $\SO_{7}(F)$
with Levi subgroup isomorphic to $\GL_{k}(F) \times \SO_{7-2k}(F)$. 
Then by \cite[Theorems 4.2, 4.3]{At}, 
with $\phi' = S_2 \oplus S_2$ and $\ep'(S_2) = -1$, 
we have 
\begin{talign*}
\Jac_{P_1}(\pi) &= |\cdot|^{\half{3}} \otimes \pi(\phi',\ep'), \\
\Jac_{P_2}(\pi) &= \Delta[\half{3},\half{1}] \otimes L(|\cdot|^{-\half{1}}; \1_{\SO_1(F)}), \\
\Jac_{P_3}(\pi) &= \Delta[\half{3},-\half{1}] \otimes \1_{\SO_1(F)}.
\end{talign*}
Combining with \cite[9.5 Proposition]{Z}, 
we see that $\Jac_P(\pi)$ is irreducible for every parabolic subgroup $P$ of $\SO_7(F)$.
\end{ex}

\subsection{Aubert duality}\label{sec:aubert}
Recall from \cite{Au} that for $\pi \in \Irr(G_n)$, 
there exists a sign $\epsilon \in \{\pm1\}$ such that 
the virtual representation
\[
\hat\pi \coloneqq \epsilon \sum_{P=MN} (-1)^{\dim A_M} \Ind_P^{G_n} (\Jac_{P}(\pi))
\]
is in fact an irreducible representation, 
where $P=MN$ runs over all standard parabolic subgroups of $G_n$, 
and $A_M$ is the maximal split torus of the center of $M$. 
We call $\hat\pi$ the \emph{Aubert dual} of $\pi$. 
\par

An explicit formula for the Aubert duals of strongly positive discrete series representations 
(\resp irreducible representations with (irreducible) $A$-parameters)
was given in \cite[Theorem 3.6]{Ma1} (\resp \cite[Theorem 6.2]{At2}). 
Using the graph $\EE(\LL_\rho)$, 
we can extend these explicit formulas to the case of ladder representations. 
\par

\begin{thm}\label{aubert}
For a ladder datum $\LL$, 
we define another ladder datum $\hat\LL$ so that 
for $\rho \in \Cusp^\GL$ with $\rho^\vee \cong \rho$, 
the graphs $\EE(\LL_\rho) = (\VV(\LL_\rho), \less, f_\rho)$ 
and $\EE(\hat\LL_\rho) = (\VV(\hat\LL_\rho), \less, \hat{f}_{\rho})$ are related as follows. 

\begin{enumerate}
\item
$(\VV(\hat\LL_\rho), \less)$ is given from $(\VV(\LL_\rho), \less)$ 
by changing the diagonal arrows and the horizontal arrows. 
More precisely, 
\begin{itemize}
\item
if $\rho$ is good, then for $k \in \Z$, 
the line $x+y = k$ (\resp $y=k$) in $\VV(\LL_\rho)$ corresponds to 
the line $y = k$ (\resp $x+y = k$) in $\VV(\hat\LL_\rho)$; 

\item
if $\rho|\cdot|^{\half{1}}$ is good, then for $k \in \Z$, 
the line $x+y = k-\half{1}\eta$ (\resp $y=k$) in $\VV(\LL_\rho)$ corresponds to 
the line $y = k$ (\resp $x+y = k+\half{1}\eta$) in $\VV(\hat\LL_\rho)$.
\end{itemize}

\item
$\hat{f}_\rho \colon \VV(\hat\LL_\rho) \rightarrow \{-1,0,1\}$ 
is given by 
\[
\hat{f}_\rho(\hat{x},\hat{y}) = \left\{
\begin{aligned}
&f_\rho(x,y) \iif \text{$\rho$ is good}, \\
&-f_\rho(x,y) \iif \text{$\rho|\cdot|^{1/2}$ is good}, 
\end{aligned}
\right. 
\]
where $(\hat{x},\hat{y}) \in \VV(\hat\LL_\rho)$ is the vertex 
corresponding to $(x,y) \in \VV(\LL_\rho)$ by (1).
\end{enumerate}
Then the Aubert dual of $\pi(\LL)$ is isomorphic to $\pi(\hat\LL)$. 
\end{thm}
\begin{proof}
This follows from \cite[Algorithm 4.1]{AM} together with Proposition \ref{der}. 
\end{proof}

\begin{ex}
Consider $\EE(\LL_\rho)$ as follows:
\begin{center}
\begin{tikzpicture}
\draw[help lines, step=1,dashed] (1,0) grid (9,2);
\node at (1,2) [above] {$-4$};
\node at (2,2) [above] {$-3$};
\node at (3,2) [above] {$-2$};
\node at (4,2) [above] {$-1$};
\node at (5,2) [above] {$0$};
\node at (6,2) [above] {$1$};
\node at (7,2) [above] {$2$};
\node at (8,2) [above] {$3$};
\node at (9,2) [above] {$4$};

\filldraw [fill = white] 
(5,2) circle (2pt) 
(4,2) circle (2pt) 
(3,2) circle (2pt) 
(2,2) circle (2pt) 
(1,2) circle (2pt) ;
\draw[->](4+0.07,2)--(5-0.07,2);
\draw[->](3+0.07,2)--(4-0.07,2);
\draw[->](2+0.07,2)--(3-0.07,2);
\draw[->](1+0.07,2)--(2-0.07,2);

\fill
(5,1) circle (2pt) ;
\node at (5,1) [above right, inner sep=1] {\scriptsize $+$};
\filldraw [fill = white] 
(6,1) circle (2pt) 
(4,1) circle (2pt) ;
\draw[->](5+0.07,1)--(6-0.07,1);
\draw[->](4+0.07,1)--(5-0.07,1);

\filldraw [fill = white] 
(9,0) circle (2pt) 
(8,0) circle (2pt) 
(7,0) circle (2pt) 
(6,0) circle (2pt) 
(5,0) circle (2pt) ;
\draw[->](8+0.07,0)--(9-0.07,0);
\draw[->](7+0.07,0)--(8-0.07,0);
\draw[->](6+0.07,0)--(7-0.07,0);
\draw[->](5+0.07,0)--(6-0.07,0);

\draw[->](6-0.05,1+0.05)--(5+0.05,2-0.05);
\draw[->](5-0.05,1+0.05)--(4+0.05,2-0.05);
\draw[->](4-0.05,1+0.05)--(3+0.05,2-0.05);

\draw[->](7-0.05,0+0.05)--(6+0.05,1-0.05);
\draw[->](6-0.05,0+0.05)--(5+0.05,1-0.05);
\draw[->](5-0.05,0+0.05)--(4+0.05,1-0.05);
\end{tikzpicture}
\end{center}
Then $\EE(\hat\LL_\rho)$ is as follows:
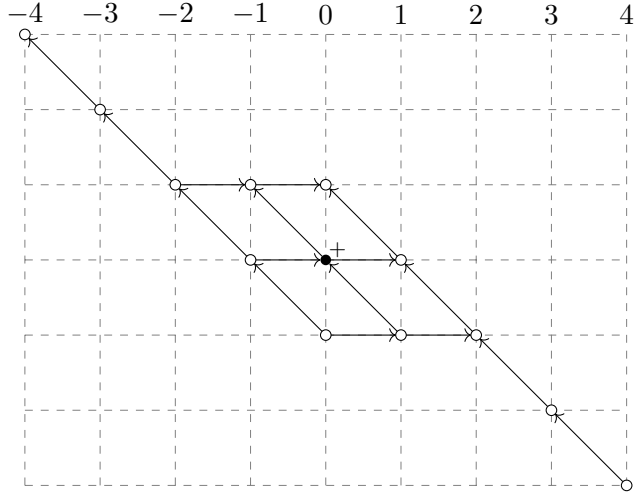
\begin{center}
\begin{tikzpicture}
\draw[help lines, step=1,dashed] (1,-2) grid (9,4);
\node at (1,4) [above] {$-4$};
\node at (2,4) [above] {$-3$};
\node at (3,4) [above] {$-2$};
\node at (4,4) [above] {$-1$};
\node at (5,4) [above] {$0$};
\node at (6,4) [above] {$1$};
\node at (7,4) [above] {$2$};
\node at (8,4) [above] {$3$};
\node at (9,4) [above] {$4$};

\filldraw [fill = white] 
(5,2) circle (2pt) 
(4,2) circle (2pt) 
(3,2) circle (2pt) 
(2,3) circle (2pt) 
(1,4) circle (2pt) ;
\draw[->](4+0.07,2)--(5-0.07,2);
\draw[->](3+0.07,2)--(4-0.07,2);

\fill
(5,1) circle (2pt) ;
\node at (5,1) [above right, inner sep=1] {\scriptsize $+$};
\filldraw [fill = white] 
(6,1) circle (2pt) 
(4,1) circle (2pt) ;
\draw[->](5+0.07,1)--(6-0.07,1);
\draw[->](4+0.07,1)--(5-0.07,1);

\filldraw [fill = white] 
(9,-2) circle (2pt) 
(8,-1) circle (2pt) 
(7,0) circle (2pt) 
(6,0) circle (2pt) 
(5,0) circle (2pt) ;
\draw[->](6+0.07,0)--(7-0.07,0);
\draw[->](5+0.07,0)--(6-0.07,0);

\draw[->](6-0.05,1+0.05)--(5+0.05,2-0.05);
\draw[->](5-0.05,1+0.05)--(4+0.05,2-0.05);
\draw[->](4-0.05,1+0.05)--(3+0.05,2-0.05);

\draw[->](7-0.05,0+0.05)--(6+0.05,1-0.05);
\draw[->](6-0.05,0+0.05)--(5+0.05,1-0.05);
\draw[->](5-0.05,0+0.05)--(4+0.05,1-0.05);

\draw[->](9-0.05,-2+0.05)--(8+0.05,-1-0.05);
\draw[->](8-0.05,-1+0.05)--(7+0.05,0-0.05);

\draw[->](3-0.05,2+0.05)--(2+0.05,3-0.05);
\draw[->](2-0.05,3+0.05)--(1+0.05,4-0.05);

\end{tikzpicture}
\end{center}
In particular, if $X_{\rho'} = \emptyset$ for any $\rho' \not\cong \rho$, 
then 
\begin{align*}
\pi(\LL) &= L(\Delta_\rho[0,-4]; \pi(\rho \boxtimes S_3, +)), \\
\pi(\hat\LL) &= L(\rho|\cdot|^{-4}, \rho|\cdot|^{-3}, \Delta_\rho[0,-2]; \pi(\rho \boxtimes S_3, +)).
\end{align*}
One can check that $\hat\pi(\LL) \cong \pi(\hat\LL)$ by \cite[Algorithm 4.1]{AM}. 
\end{ex}

\section{Determinantal formula}\label{sec:det}
In this section, we describe ladder representations of $G_n$ 
in terms of linear combinations of standard modules in $\Rr^G$.

\subsection{Statement}
Let $\LL = (\ub{X}, \ub{l}, \ub{\eta})$ be a ladder datum of $G_n$. 
For $\rho \in \Cusp^\GL$ with $\rho^\vee \cong \rho$, 
we write $X_\rho = \{\rho|\cdot|^{x_1},\dots,\rho|\cdot|^{x_{t_\rho}}\}$ 
with $x_1<\dots<x_{t_\rho}$.
For a non-negative integer $t$, we denote by $\Sy_t$
the symmetric group of order $t!$. 
In particular, $\Sy_0 = \Sy_1$ is the trivial group.

\begin{defi}\label{def:sgn}
\begin{enumerate}
\item
Let $\Sy(\LL)$ be the subset of the direct sum of symmetric groups $\oplus_{\rho} \Sy_{t_\rho}$
consisting of $\sigma = (\sigma_\rho)_\rho$ such that
\begin{itemize}
\item
for $1 \leq i < j \leq l_\rho$, we have $\sigma_\rho(i) < \sigma_\rho(j)$; 
\item
for $l_\rho+1 \leq i < j \leq t_\rho-l_\rho$, we have $\sigma_\rho(i) < \sigma_\rho(j)$; 
\item
if $x_i \leq -1$, or if $x_i = -1/2$ and $\eta_\rho = -1$, 
then $\sigma_\rho^{-1}(i) \leq l_\rho$. 
\end{itemize}

\item
The restriction of the product of the sign characters 
\[
\bigoplus_{\rho} \Sy_{t_\rho} \xrightarrow{\oplus \sgn} \bigoplus_\rho \{\pm1\} 
\xrightarrow{\text{product}} \{\pm1\}
\]
gives a sign map
\[
\sgn \colon \Sy(\LL) \rightarrow \{\pm1\}.
\]

\item
For $\sigma = (\sigma_\rho)_\rho \in \Sy(\LL)$, 
set 
\begin{align*}
J_{\sigma,\rho}^+ &\coloneqq \{1 \leq j \leq l_\rho \;|\; \sigma_\rho(j) < \sigma_\rho(t_\rho-j+1)\}, \\
J_{\sigma,\rho}^- &\coloneqq \{1 \leq j \leq l_\rho \;|\; \sigma_\rho(j) > \sigma_\rho(t_\rho-j+1)\}.
\end{align*}
We define a direct sum of standard modules $I_\sigma(\LL)$ by 
\[
I_\sigma(\LL) \coloneqq 
\left(
\bigtimes_\rho \bigtimes_{j \in J_{\sigma,\rho}^+} 
\Delta_\rho[x_{\sigma_\rho(j)}, -x_{\sigma_\rho(t_\rho-j+1)}] 
\right)
\rtimes \left(\bigoplus_{\delta}\pi(\phi_\sigma, \ep_{\sigma, \delta})\right), 
\]
where $\delta = (\delta_\rho)_\rho$ runs over the set of tuples of maps 
\[
\delta_\rho \colon J_{\sigma,\rho}^- \rightarrow \{\pm1\},
\]
and $(\phi_\sigma, \ep_{\sigma,\delta})$ is given by
\begin{align*}
\phi_\sigma \coloneqq &\bigoplus_\rho
\bigoplus_{i=l_\rho+1}^{t_\rho-l_\rho} \rho \boxtimes S_{2x_{\sigma_\rho(i)}+1}
\\&\oplus
\bigoplus_\rho\bigoplus_{j \in J_{\sigma,\rho}^-}
\rho \boxtimes (S_{2x_{\sigma_\rho(j)}+1} \oplus S_{2x_{\sigma_\rho(t_\rho-j+1)}+1})
\end{align*}
and 
\begin{align*}
\ep_{\sigma,\delta}(\rho \boxtimes S_{2x_{\sigma_\rho(i)+1}}) 
&\coloneqq (-1)^{i-l_\rho-1}\eta_\rho, \\
\ep_{\sigma,\delta}(\rho \boxtimes S_{2x_{\sigma_\rho(j)}+1}) 
&=
\ep_{\sigma,\delta}(\rho \boxtimes S_{2x_{\sigma_\rho(t_\rho-j+1)}+1}) 
\coloneqq \delta_\rho(j)
\end{align*}
for $l_\rho+1 \leq i \leq t_\rho-l_\rho$ and $j \in J_{\sigma,\rho}^-$.
Here, if $\phi_\sigma \supset \rho \boxtimes S_0$ and $\ep_{\sigma,\delta}(\rho \boxtimes S_0) = -1$, 
then we interpret $\pi(\phi_\sigma, \ep_{\sigma,\delta})$ to be $0$.
\end{enumerate}
\end{defi}

For a ladder datum $\LL$ for $G_n$, we have a ladder representation $\pi(\LL) \in \Irr(G_n)$. 
Recall in \S \ref{rep}, 
for a cuspidal support $\s$ of $G_n$, 
we have a projection operator $p_\s \colon \Rr(G_n) \rightarrow \Rr(G_n)$. 
If $\Pi$ is a standard module of $G_n$, 
since all irreducible subquotients of $\Pi$ have the same cuspidal support, 
one has $p_\s([\Pi]) = [\Pi]$ or $p_\s([\Pi]) = 0$. 
By \cite[Theorem 4.2]{At}, one can determine $p_\s([\Pi])$. 
\par

Now we can state our main result, 
which is an analogue of Tadi\'c's determinantal formula \cite[Theorem 1 (ii)]{LM}.

\begin{thm}\label{determinantal}
Let $\LL$ be a ladder datum for $G_n$. 
Then 
\[
[\pi(\LL)] = p_\s \left(\sum_{\sigma \in \Sy(\LL)} \sgn(\sigma) [I_\sigma(\LL)] \right), 
\]
where $\s = \supp(\pi(\LL))$.
\end{thm}

\subsection{Examples}
In this subsection, except for Example \ref{counter2} below, 
we assume that $\LL = (\ub{X},\ub{l},\ub{\eta})$ satisfies that 
$X_\rho = \emptyset$ unless $\rho = \1_{\GL_1(F)}$. 
We identify $\LL$ with $\LL_{\1_{\GL_1(F)}}$ and drop $\1_{\GL_1(F)}$ from the subscript. 
For example, $[x,y] = \{|\cdot|^x,|\cdot|^{x-1},\dots,|\cdot|^y\}$ is a segment, 
and $\Delta[x,y]$ is the associated Steinberg representation.
\par

Once $\rho$ is fixed, 
for $\phi = \oplus_{i=1}^r \rho \boxtimes S_{2x_i+1}$ 
and $\epsilon_i = \ep(\rho \boxtimes S_{2x_i+1})$, 
we write
\[
\pi(x_1^{\epsilon_1},\dots,x_r^{\epsilon_r}) \coloneqq \pi(\phi,\ep).
\]

\begin{ex}\label{ex1}
Let us consider the trivial representation $\1_{\SO_5(F)}$ of $\SO_5(F)$.
Note that $\1_{\SO_5(F)} = \pi(\LL)$ with $\LL \coloneqq ([3/2,-3/2],2,-1)$.
By Definition \ref{def:sgn} (1), we have
\[
\Sy(\LL) = \{\sigma \in \Sy_4 \;|\; \sigma(1) = 1\}.
\]
We give the list of $\sigma \in \Sy(\LL)$, $\sgn(\sigma)$ and $I_\sigma(\LL)$
in Table \ref{tab1}. 
Here, we recall that $\Delta[-\half{3},-\half{1}] = \1_{\GL_0(F)}$, $\Delta[-\half{3},\half{1}] = 0$
and $\pi((-\half{1})^-,(\half{1})^-) = \pi((-\half{1})^-,(\half{3})^-) = 0$. 

{\renewcommand\arraystretch{1.3}
\begin{longtable}[c]{ccc}
\caption{List of $\sigma$, $\sgn(\sigma)$, $I_\sigma(\LL)$ in Example \ref{ex1}.}
\label{tab1}
\\
\hline
$\sigma$ & $\sgn(\sigma)$ & $I_\sigma(\LL)$ \\
\hline
$\left(\begin{smallmatrix}
1&2&3&4\\
1&2&3&4
\end{smallmatrix}\right)$
& $1$ &
$|\cdot|^{-\half{3}} \times |\cdot|^{-\half{1}} \rtimes \1_{\SO_1(F)}$\\
$\left(\begin{smallmatrix}
1&2&3&4\\
1&2&4&3
\end{smallmatrix}\right)$
& $-1$ &
$\Delta[-\half{1},-\half{3}] \rtimes \1_{\SO_1(F)}$\\
$\left(\begin{smallmatrix}
1&2&3&4\\
1&3&2&4
\end{smallmatrix}\right)$
& $-1$ &
$|\cdot|^{-\half{3}} \rtimes \pi((\half{1})^+)$\\
$\left(\begin{smallmatrix}
1&2&3&4\\
1&3&4&2
\end{smallmatrix}\right)$
& $1$ & $0$\\
$\left(\begin{smallmatrix}
1&2&3&4\\
1&4&2&3
\end{smallmatrix}\right)$
& $1$ &
$\pi((\half{3})^+)$\\
$\left(\begin{smallmatrix}
1&2&3&4\\
1&4&3&2
\end{smallmatrix}\right)$
& $-1$ & $0$
\end{longtable}
}

Then Theorem \ref{determinantal} says that 
\begin{talign*}
\Big[\1_{\SO_5(F)}\Big] &=
\Big[|\cdot|^{-\half{3}} \times |\cdot|^{-\half{1}} \rtimes \1_{\SO_1(F)}\Big]
-\Big[\Delta[-\half{1},-\half{3}] \rtimes \1_{\SO_1(F)}\Big]
\\&\quad
-\Big[|\cdot|^{-\half{3}} \rtimes \pi((\half{1})^+)\Big]
+\Big[\pi((\half{3})^+)\Big]. 
\end{talign*}
\end{ex}

\begin{ex}\label{ex2}
Let us consider $\LL \coloneqq ([3,-1],2,+1)$. 
Then 
\[
\pi(\LL) = L(\Delta[-1,-3], \Delta[0,-2]; \pi(1^+))
\]
is an irreducible representation of $\Sp_{14}(F)$.
By Definition \ref{def:sgn} (1), we have
\[
\Sy(\LL) = \{\sigma \in \Sy_5 \;|\; \sigma(1) = 1\}.
\]
The list of $\sigma \in \Sy(\LL)$, $\sgn(\sigma)$ and $I_\sigma(\LL)$ is given in Table \ref{tab2}.
Here, we recall that $\Delta[-1,0] = \1_{\GL_0(F)}$.

{\renewcommand\arraystretch{1.3}
\begin{longtable}[c]{ccc}
\caption{List of $\sigma$, $\sgn(\sigma)$, $I_\sigma(\LL)$ in Example \ref{ex2}.}
\label{tab2}
\\
\hline
$\sigma$ & $\sgn(\sigma)$ & $I_\sigma(\LL)$ \\
\hline
$\left(\begin{smallmatrix}
1&2&3&4&5\\
1&2&3&4&5
\end{smallmatrix}\right)$
& $1$ &
$\Delta[-1,-3] \times \Delta[0,-2] \rtimes \pi(1^+)$\\
$\left(\begin{smallmatrix}
1&2&3&4&5\\
1&2&3&5&4
\end{smallmatrix}\right)$
& $-1$ &
$\Delta[-1,-2] \times \Delta[0,-3] \rtimes \pi(1^+)$\\
$\left(\begin{smallmatrix}
1&2&3&4&5\\
1&2&4&3&5
\end{smallmatrix}\right)$
& $-1$ &
$\Delta[-1,-3] \times \Delta[0,-1] \rtimes \pi(2^+)$\\
$\left(\begin{smallmatrix}
1&2&3&4&5\\
1&2&4&5&3
\end{smallmatrix}\right)$
& $1$ &
$|\cdot|^{-1} \times \Delta[0,-3] \rtimes \pi(2^+)$\\
$\left(\begin{smallmatrix}
1&2&3&4&5\\
1&2&5&3&4
\end{smallmatrix}\right)$
& $1$ &
$\Delta[-1,-2] \times \Delta[0,-1] \rtimes \pi(3^+)$\\
$\left(\begin{smallmatrix}
1&2&3&4&5\\
1&2&5&4&3
\end{smallmatrix}\right)$
& $-1$ &
$|\cdot|^{-1} \times \Delta[0,-2] \rtimes \pi(3^+)$\\
$\left(\begin{smallmatrix}
1&2&3&4&5\\
1&3&2&4&5
\end{smallmatrix}\right)$
& $-1$ &
$\Delta[-1,-3] \times \Delta[1,-2] \rtimes \pi(0^+)$\\
$\left(\begin{smallmatrix}
1&2&3&4&5\\
1&3&2&5&4
\end{smallmatrix}\right)$
& $1$ &
$\Delta[-1,-2] \times \Delta[1,-3] \rtimes \pi(0^+)$\\
$\left(\begin{smallmatrix}
1&2&3&4&5\\
1&3&4&2&5
\end{smallmatrix}\right)$
& $1$ &
$\Delta[-1,-3] \rtimes (\pi(0^+,1^+,2^+) \oplus \pi(0^-,1^-,2^+))$\\
$\left(\begin{smallmatrix}
1&2&3&4&5\\
1&3&4&5&2
\end{smallmatrix}\right)$
& $-1$ &
$\Delta[1,-3] \rtimes \pi(2^+)$\\
$\left(\begin{smallmatrix}
1&2&3&4&5\\
1&3&5&2&4
\end{smallmatrix}\right)$
& $-1$ &
$\Delta[-1,-2] \rtimes (\pi(0^+,1^+,3^+)\oplus\pi(0^-,1^-,3^+))$\\
$\left(\begin{smallmatrix}
1&2&3&4&5\\
1&3&5&4&2
\end{smallmatrix}\right)$
& $1$ &
$\Delta[1,-2] \rtimes \pi(3^+)$\\
$\left(\begin{smallmatrix}
1&2&3&4&5\\
1&4&2&3&5
\end{smallmatrix}\right)$
& $1$ &
$\Delta[-1,-3] \rtimes (\pi(0^+,1^+,2^+) \oplus \pi(0^+,1^-,2^-))$\\
$\left(\begin{smallmatrix}
1&2&3&4&5\\
1&4&2&5&3
\end{smallmatrix}\right)$
& $-1$ &
$|\cdot|^{-1} \times \Delta[2,-3] \rtimes \pi(0^+)$\\
$\left(\begin{smallmatrix}
1&2&3&4&5\\
1&4&3&2&5
\end{smallmatrix}\right)$
& $-1$ &
$\Delta[-1,-3] \rtimes (\pi(0^+,1^+,2^+) \oplus \pi(0^-,1^+,2^-))$\\
$\left(\begin{smallmatrix}
1&2&3&4&5\\
1&4&3&5&2
\end{smallmatrix}\right)$
& $1$ &
$\Delta[2,-3] \rtimes \pi(1^+)$\\
$\left(\begin{smallmatrix}
1&2&3&4&5\\
1&4&5&2&3
\end{smallmatrix}\right)$
& $1$ &
$|\cdot|^{-1} \rtimes (\pi(0^+,2^+,3^+) \oplus \pi(0^-,2^-,3^+))$\\
$\left(\begin{smallmatrix}
1&2&3&4&5\\
1&4&5&3&2
\end{smallmatrix}\right)$
& $-1$ &
$\pi(1^+,2^+,3^+) \oplus \pi(1^-,2^-,3^+)$\\
$\left(\begin{smallmatrix}
1&2&3&4&5\\
1&5&2&3&4
\end{smallmatrix}\right)$
& $-1$ &
$\Delta[-1,-2] \rtimes (\pi(0^+,1^+,3^+) \oplus \pi(0^+,1^-,3^-))$\\
$\left(\begin{smallmatrix}
1&2&3&4&5\\
1&5&2&4&3
\end{smallmatrix}\right)$
& $1$ &
$|\cdot|^{-1} \rtimes (\pi(0^+,2^+,3^+) \oplus \pi(0^+,2^-,3^-))$\\
$\left(\begin{smallmatrix}
1&2&3&4&5\\
1&5&3&2&4
\end{smallmatrix}\right)$
& $1$ &
$\Delta[-1,-2] \rtimes (\pi(0^+,1^+,3^+) \oplus \pi(0^-,1^+,3^-))$\\
$\left(\begin{smallmatrix}
1&2&3&4&5\\
1&5&3&4&2
\end{smallmatrix}\right)$
& $-1$ &
$\pi(1^+,2^+,3^+) \oplus \pi(1^+,2^-,3^-)$\\
$\left(\begin{smallmatrix}
1&2&3&4&5\\
1&5&4&2&3
\end{smallmatrix}\right)$
& $-1$ &
$|\cdot|^{-1} \rtimes (\pi(0^+,2^+,3^+) \oplus \pi(0^-,2^+,3^-))$\\
$\left(\begin{smallmatrix}
1&2&3&4&5\\
1&5&4&3&2
\end{smallmatrix}\right)$
& $1$ &
$\pi(1^+,2^+,3^+) \oplus \pi(1^-,2^+,3^-)$
\end{longtable}
}

Among the standard modules appearing in this list, 
exactly $4$ representations
\begin{align*}
\Delta[-1,-3] \rtimes \pi(0^-,1^+,2^-), 
&\quad
\Delta[-1,-2] \rtimes \pi(0^-,1^+,3^-), 
\\
|\cdot|^{-1} \rtimes \pi(0^-,2^+,3^-), 
&\quad
\pi(1^-,2^+,3^-)
\end{align*}
are killed by $p_\s([\cdot])$ with $\s = \supp(\pi(\LL))$.
In addition, the $4$ representations
\begin{align*}
\Delta[-1,-3] \rtimes \pi(0^+,1^+,2^+), 
&\quad
\Delta[-1,-2] \rtimes \pi(0^+,1^+,3^+), 
\\
|\cdot|^{-1} \rtimes \pi(0^+,2^+,3^+), 
&\quad
\pi(1^+,2^+,3^+)
\end{align*}
appear three times, 
but two of them are cancelled, respectively. 
Therefore, by Theorem \ref{determinantal}, 
we conclude that $[\pi(\LL)]$ is a linear combination of 
exactly $24$ standard modules with coefficients in $\{\pm1\}$. 
\end{ex}

\begin{ex}\label{ex3}
Let us consider $\LL \coloneqq ([4,0],1,-1)$.
Then 
\[
\pi(\LL) = L(\Delta[0,-4]; \pi(1^-,2^+,3^-))
\]
is an irreducible representation of $\Sp_{24}(F)$.
Any element $\sigma \in \Sy(\LL) \subset \Sy_5$ 
is determined by the pair $(\sigma(1),\sigma(5))$
since $\sigma(2) < \sigma(3) < \sigma(4)$.
Hence $|\Sy(\LL)| = 20$. 
We list $\sigma \in \Sy(\LL)$, $\sgn(\sigma)$ and $I_\sigma(\LL)$ in Table \ref{tab3}.

{\renewcommand\arraystretch{1.3}
\begin{longtable}[c]{ccc}
\caption{List of $\sigma$, $\sgn(\sigma)$, $I_\sigma(\LL)$ in Example \ref{ex3}.}
\label{tab3}
\\
\hline
$\sigma$ & $\sgn(\sigma)$ & $I_\sigma(\LL)$ \\
\hline
$\left(\begin{smallmatrix}
1&2&3&4&5\\
1&2&3&4&5
\end{smallmatrix}\right)$
& $1$ &
$\Delta[0,-4] \rtimes \pi(1^-,2^+,3^-)$\\
$\left(\begin{smallmatrix}
1&2&3&4&5\\
1&2&3&5&4
\end{smallmatrix}\right)$
& $-1$ &
$\Delta[0,-3] \rtimes \pi(1^-,2^+,4^-)$\\
$\left(\begin{smallmatrix}
1&2&3&4&5\\
1&2&4&5&3
\end{smallmatrix}\right)$
& $1$ &
$\Delta[0,-2] \rtimes \pi(1^-,3^+,4^-)$\\
$\left(\begin{smallmatrix}
1&2&3&4&5\\
1&3&4&5&2
\end{smallmatrix}\right)$
& $-1$ &
$\Delta[0,-1] \rtimes \pi(2^-,3^+,4^-)$\\
$\left(\begin{smallmatrix}
1&2&3&4&5\\
2&1&3&4&5
\end{smallmatrix}\right)$
& $-1$ &
$\Delta[1,-4] \rtimes \pi(0^-,2^+,3^-)$\\
$\left(\begin{smallmatrix}
1&2&3&4&5\\
2&1&3&5&4
\end{smallmatrix}\right)$
& $1$ &
$\Delta[1,-3] \rtimes \pi(0^-,2^+,4^-)$\\
$\left(\begin{smallmatrix}
1&2&3&4&5\\
2&1&4&5&3
\end{smallmatrix}\right)$
& $-1$ &
$\Delta[1,-2] \rtimes \pi(0^-,3^+,4^-)$\\
$\left(\begin{smallmatrix}
1&2&3&4&5\\
2&3&4&5&1
\end{smallmatrix}\right)$
& $1$ &
$\pi(0^+,1^+,2^-,3^+,4^-) \oplus \pi(0^-,1^-,2^-,3^+,4^-)$\\
$\left(\begin{smallmatrix}
1&2&3&4&5\\
3&1&2&4&5
\end{smallmatrix}\right)$
& $1$ &
$\Delta[2,-4] \rtimes \pi(0^-,1^+,3^-)$\\
$\left(\begin{smallmatrix}
1&2&3&4&5\\
3&1&2&5&4
\end{smallmatrix}\right)$
& $-1$ &
$\Delta[2,-3] \rtimes \pi(0^-,1^+,4^-)$\\
$\left(\begin{smallmatrix}
1&2&3&4&5\\
3&1&4&5&2
\end{smallmatrix}\right)$
& $1$ &
$\pi(0^-,1^+,2^+,3^+,4^-) \oplus \pi(0^-,1^-,2^-,3^+,4^-)$\\
$\left(\begin{smallmatrix}
1&2&3&4&5\\
3&2&4&5&1
\end{smallmatrix}\right)$
& $-1$ &
$\pi(0^+,1^-,2^+,3^+,4^-) \oplus \pi(0^-,1^-,2^-,3^+,4^-)$\\
$\left(\begin{smallmatrix}
1&2&3&4&5\\
4&1&2&3&5
\end{smallmatrix}\right)$
& $-1$ &
$\Delta[3,-4] \rtimes \pi(0^-,1^+,2^-)$\\
$\left(\begin{smallmatrix}
1&2&3&4&5\\
4&1&2&5&3
\end{smallmatrix}\right)$
& $1$ &
$\pi(0^-,1^+,2^+,3^+,4^-) \oplus \pi(0^-,1^+,2^-,3^-,4^-)$\\
$\left(\begin{smallmatrix}
1&2&3&4&5\\
4&1&3&5&2
\end{smallmatrix}\right)$
& $-1$ &
$\pi(0^-,1^+,2^+,3^+,4^-) \oplus \pi(0^-,1^-,2^+,3^-,4^-)$\\
$\left(\begin{smallmatrix}
1&2&3&4&5\\
4&2&3&5&1
\end{smallmatrix}\right)$
& $1$ &
$\pi(0^+,1^-,2^+,3^+,4^-) \oplus \pi(0^-,1^-,2^+,3^-,4^-)$\\
$\left(\begin{smallmatrix}
1&2&3&4&5\\
5&1&2&3&4
\end{smallmatrix}\right)$
& $1$ &
$\pi(0^-,1^+,2^-,3^+,4^+) \oplus \pi(0^-,1^+,2^-,3^-,4^-)$\\
$\left(\begin{smallmatrix}
1&2&3&4&5\\
5&1&2&4&3
\end{smallmatrix}\right)$
& $-1$ &
$\pi(0^-,1^+,2^+,3^-,4^+) \oplus \pi(0^-,1^+,2^-,3^-,4^-)$\\
$\left(\begin{smallmatrix}
1&2&3&4&5\\
5&1&3&4&2
\end{smallmatrix}\right)$
& $1$ &
$\pi(0^-,1^+,2^+,3^-,4^+) \oplus \pi(0^-,1^-,2^+,3^-,4^-)$\\
$\left(\begin{smallmatrix}
1&2&3&4&5\\
5&2&3&4&1
\end{smallmatrix}\right)$
& $-1$ &
$\pi(0^+,1^-,2^+,3^-,4^+) \oplus \pi(0^-,1^-,2^+,3^-,4^-)$
\end{longtable}
}

If we put $\s = \supp(\pi(\LL))$, we note that 
\begin{align*}
&p_\s([\pi(0^+,1^-,2^+,3^+,4^-) \oplus \pi(0^-,1^-,2^+,3^-,4^-)]) = 0, \\
&p_\s([\pi(0^-,1^+,2^+,3^-,4^+) \oplus \pi(0^-,1^-,2^+,3^-,4^-)]) = 0, \\
&p_\s([\pi(0^+,1^-,2^+,3^-,4^+) \oplus \pi(0^-,1^-,2^+,3^-,4^-)]) = 0, 
\end{align*}
and
\begin{align*}
&p_\s([\pi(0^-,1^+,2^+,3^+,4^-) \oplus \pi(0^-,1^-,2^-,3^+,4^-)])
\\&-p_\s([\pi(0^+,1^-,2^+,3^+,4^-) \oplus \pi(0^-,1^-,2^-,3^+,4^-)])
\\&+p_\s([\pi(0^-,1^+,2^+,3^+,4^-) \oplus \pi(0^-,1^+,2^-,3^-,4^-)])
\\&-p_\s([\pi(0^-,1^+,2^+,3^-,4^+) \oplus \pi(0^-,1^+,2^-,3^-,4^-)])
\\&-p_\s([\pi(0^-,1^+,2^+,3^+,4^-) \oplus \pi(0^-,1^-,2^+,3^-,4^-)])
\\&= [\pi(0^-,1^+,2^+,3^+,4^-)].
\end{align*}
Here, in the left hand side of the last equation, 
we note that $[\pi(0^-,1^+,2^+,3^+,4^-)]$ appears three times. 
In conclusion, Theorem \ref{determinantal} tells us that 
$[\pi(\LL)]$ is an alternating sum of exactly $15$ standard modules. 
\end{ex}

\begin{rem}
As we have seen these examples, 
one might expect that $[\pi(\LL)]$ would be a linear combination of 
standard modules with coefficients in $\{\pm1\}$. 
\end{rem}

As in Example \ref{counter},
there are irreducible representations $\pi$
which are not ladder but all Jacquet modules of $\pi$ are completely reducible. 
For such representations, our determinantal formula may fail. 

\begin{ex}\label{counter2}
In this example, we fix an irreducible cuspidal symplectic representation $\rho$ of $\GL_2(F)$. 
Consider 
\[
\pi \coloneqq L(\Delta_\rho[-1,-2]; \pi(0^+,1^+)) \in \Irr(\SO_{17}(F)). 
\]
One can compute all Jacquet modules of $\pi$ 
by the argument of \cite[Theorem 4.3]{At}
together with \cite[Section 6]{X1}, \cite[Theorem 4.2]{At} and \cite[Theorem 7.1]{AM}.
To be precise, if $P_{2k}$ is the standard parabolic subgroup of $\SO_{17}(F)$ 
with Levi subgroup isomorphic to $\GL_{2k}(F) \times \SO_{17-2k}(F)$, 
then 
\begin{align*}
\Jac_{P_2}(\pi) &= 
\Big(\rho|\cdot|^{-1} \otimes L(\rho|\cdot|^{-2}; \pi(0^+,1^+))\Big)
\\&\quad\oplus 
\Big(\rho|\cdot|^{1} \otimes L(\Delta_\rho[-1,-2]; \pi(0^+,0^+))\Big),
\\
\Jac_{P_4}(\pi) &= 
\Big(\Delta_\rho[-1,-2] \otimes \pi(0^+,1^+)\Big)
\\&\quad\oplus 
\Big(L(\rho|\cdot|^{-1}, \rho|\cdot|^{1}) \otimes L(\rho|\cdot|^{-2}; \pi(0^+,0^+))\Big)
\\&\quad\oplus 
\Big(L(\rho|\cdot|^{1}, \rho|\cdot|^2) \otimes L(\rho|\cdot|^{-1}; \pi(0^+,0^+))\Big),
\\
\Jac_{P_6}(\pi) &= 
\Big(L(\Delta_\rho[-1,-2], |\cdot|^1) \otimes \pi(0^+,0^+)\Big)
\\&\quad\oplus 
\Big(L(\rho|\cdot|^{-1}, \Delta_\rho[1,0]) \otimes L(\rho|\cdot|^{-2}; \1_{\SO_1(F)})\Big)
\\&\quad\oplus 
\Big(L( \rho|\cdot|^{-1}, \rho|\cdot|^{1}, \rho|\cdot|^2) \otimes \pi(0^+,0^+)\Big),
\\
\Jac_{P_8}(\pi) &= 
\Big(L(\Delta_\rho[-1,-2], \Delta_\rho[1,0]) \otimes \1_{\SO_1(F)}\Big)
\\&\quad\oplus 
\Big(L(\rho|\cdot|^{-1}, \Delta_\rho[1,0], \rho|\cdot|^2) \otimes \1_{\SO_1(F)}\Big).
\end{align*}
Since the representations of $\GL_{2k}(F)$ appearing above are all ladder, 
we see that 
$\Jac_P(\pi)$ is completely reducible for any parabolic subgroup $P$ of $\SO_{17}(F)$. 
\par

On the other hand, if we try to apply Theorem \ref{determinantal} to $\pi$ formally, 
the right hand side should be equal to 
\[
\tag{$\ast$}
\Big[\Delta_\rho[-1,-2] \rtimes \pi(0^+,1^+)\Big] 
- \Big[\rho|\cdot|^{-1} \rtimes \pi(0^+,2^+)\Big]
+ \Big[\pi(1^+,2^+)\Big].
\]
By taking $D_{\rho|\cdot|^2}$, it becomes $[\pi(1^+,1^+)]$, 
whereas $D_{\rho|\cdot|^2}([\pi]) = 0$ by \cite[Theorem 7.1]{AM}.
Hence $[\pi]$ is not equal to $(\ast)$.
\end{ex}

\subsection{Proof of Theorem \ref{determinantal}}
In this subsection, we prove Theorem \ref{determinantal}.
\par

Let $\LL = (\ub{X}, \ub{l}, \ub{\eta})$ be a ladder datum for $G_n$
and $\pi(\LL) \in \Irr(G_n)$ be the associated ladder representation.
For $\rho \in \Cusp(\GL_d(F))$ such that $\rho^\vee \cong \rho$, 
we write $X_\rho = \{\rho|\cdot|^{x_1},\dots,\rho|\cdot|^{x_{t_\rho}}\}$ with $x_1<\dots<x_{t_\rho}$.
We defined a directed $3$-colorable graph $\EE(\LL_\rho) = (\VV(\LL_\rho), \less, f_\rho)$
in Definition \ref{graph}.
Set $2m_\rho = |f_\rho^{-1}(0)|$, which is even. 
\par

We will prove Theorem \ref{determinantal} by induction on $n$. 
If $m_\rho=0$ for all $\rho$, then by Corollary \ref{sc} (1), $\pi(\LL)$ is supercuspidal.
In this case, 
since $|\Sy(\LL)| = 1$, Theorem \ref{determinantal} is trivial.
\par

From now, assume the assertion of Theorem \ref{determinantal} for $G_{n'}$ with $n' < n$.
Furthermore, we fix $\rho \in \Cusp(\GL_d(F))$ with $\rho^\vee \cong \rho$ such that $m_\rho > 0$. 
Write $t \coloneqq t_\rho$, $(l,\eta) \coloneqq (l_\rho, \eta_\rho)$ and $m \coloneqq m_\rho$
for simplicity.

\begin{lem}\label{lem1}
If $\lambda \in \R$ satisfies that $D_{\rho|\cdot|^\lambda}([\pi(\LL)]) = 0$, 
then 
\[
D_{\rho|\cdot|^{\lambda}} \circ p_\s\left(\sum_{\sigma \in \Sy(\LL)} \sgn(\sigma) [I_\sigma(\LL)] \right) = 0.
\]
\end{lem}
\begin{proof}
If $\rho|\cdot|^{\pm\lambda} \not\in \s$, 
then $D_{\rho|\cdot|^{\lambda}} \circ p_\s = 0$ so that the assertion is trivial. 
Hence we may assume that $\rho|\cdot|^{\pm\lambda} \in \s$. 
In this case, since $D_{\rho|\cdot|^{\lambda}} \circ p_\s = p_{\s'} \circ D_{\rho|\cdot|^\lambda}$
with $\s' = \s \setminus \{\rho|\cdot|^{\lambda}, \rho|\cdot|^{-\lambda}\}$, 
it suffices to show that 
$D_{\rho|\cdot|^{\lambda}}(\sum_{\sigma \in \Sy(\LL)} \sgn(\sigma) [I_\sigma(\LL)]) = 0$.
\par

To show this, 
we may assume that $D_{\rho|\cdot|^{\lambda}}([I_\sigma(\LL)]) \not= 0$ for some $\sigma \in \Sy(\LL)$. 
In this case, we have $\lambda \in \{x_1,\dots,x_t\}$ 
by Tadi\'c's formula (Proposition \ref{Tadic}) and \cite[Lemma 7.3]{X1}.
Since $D_{\rho|\cdot|^\lambda}([\pi(\LL)]) = 0$, 
by Proposition \ref{der}, 
one of the following holds: 
\begin{itemize}
\item
$\lambda = x_1$ and $l = 0$; or
\item
there is $2 \leq j \leq t$ such that 
$x_j = \lambda$ and $x_{j-1} = \lambda-1 = x_j-1$.
\end{itemize}
However, in the former case, we have $\sigma_\rho = \id$ for any $\sigma \in \Sy(\LL)$.
This implies that $D_{\rho|\cdot|^{\lambda}}([I_\sigma(\LL)]) = 0$, 
which is a contradiction.
Hence the latter case must occur.
\par

Suppose that $\sigma = (\sigma_{\rho'})_{\rho'} \in \Sy(\LL)$ 
satisfies that $D_{\rho|\cdot|^{\lambda}}([I_\sigma(\LL)]) \not= 0$. 
Then we note that one of $j$ or $j-1$ does not belong to $\sigma(\{l+1,\dots,t-l\})$. 
We will define $\sigma' = (\sigma'_{\rho'})_{\rho'}$ as follows. 
For $\rho' \not\cong \rho$, set $\sigma'_{\rho'} \coloneqq \sigma_{\rho'}$. 
To define $\sigma'_\rho$, we consider two cases. 

\begin{enumerate}
\item
If $j, j-1 \in \sigma(\{1,\dots,l\})$, 
writing $j-1 = \sigma(i)$ and $j = \sigma(i')$ with $1 \leq i < i' \leq l$, 
we set $\sigma'_\rho \coloneqq \sigma_\rho \cdot (t-i+1, t-i'+1)$.
Then $\sigma' \in \Sy(\LL)$. 
Moreover, 
there exists a direct sum of standard modules $\Pi$ satisfying $D_{\rho|\cdot|^\lambda}(\Pi) = 0$ 
such that 
\begin{align*}
I_\sigma(\LL) &= 
\Delta_\rho[\lambda-1, x_{\sigma(t-i+1)}] \times \Delta_\rho[\lambda, x_{\sigma(t-i'+1)}] 
\rtimes \Pi, \\
I_{\sigma'}(\LL) &= 
\Delta_\rho[\lambda-1, x_{\sigma(t-i'+1)}] \times \Delta_\rho[\lambda, x_{\sigma(t-i+1)}] 
\rtimes \Pi.
\end{align*}
Hence 
\[
D_{\rho|\cdot|^{\lambda}}([I_\sigma(\LL)])
=
\Delta_\rho[\lambda-1, x_{\sigma(t-i+1)}] \times \Delta_\rho[\lambda-1, x_{\sigma(t-i'+1)}] 
\rtimes \Pi
=
D_{\rho|\cdot|^{\lambda}}([I_{\sigma'}(\LL)]).
\]

\item
If one of $j$ or $j-1$ does not belong to $\sigma(\{1,\dots,l\})$, 
set $\sigma'_\rho \coloneqq (j,j-1) \cdot \sigma_\rho$. 
In this case, $\sigma' \in \Sy(\LL)$. 
We claim that  
\[
D_{\rho|\cdot|^{\lambda}}([I_\sigma(\LL)])
=
D_{\rho|\cdot|^{\lambda}}([I_{\sigma'}(\LL)]).
\]
This equation can be proven similarly to the first case 
unless there exists $1 \leq i \leq l$ such that
$(j-1,j) = (\sigma(i), \sigma(t-i+1))$ or $(j-1,j) = (\sigma(t-i+1), \sigma(i))$. 
By symmetry, let us consider the former case 
so that $i \in J_{\sigma,\rho}^+$ and $i \in J_{\sigma',\rho}^-$.
Hence we can find a product of Steinberg representations $\tau$
and a sequence of tempered representations $\pi(\phi,\ep_1), \dots, \pi(\phi,\ep_r)$ such that 
\begin{align*}
I_\sigma(\LL) &= \Delta[\lambda, -(\lambda-1)] \times \tau \rtimes \bigoplus_{i=1}^r \pi(\phi,\ep_r), \\
I_{\sigma'}(\LL) &= \tau \rtimes \bigoplus_{i=1}^r (\pi(\phi',\ep'_{r,+}) \oplus \pi(\phi',\ep'_{r,-})),
\end{align*}
where $\phi' = \phi \oplus (\rho \boxtimes S_{2\lambda-1} \oplus \rho \boxtimes S_{2\lambda+1})$
and 
\[
\ep'_{r,\pm}(\rho \boxtimes S_{2\lambda-1}) = \ep'_{r,\pm}(\rho \boxtimes S_{2\lambda+1}) = \pm1.
\]
Then by \cite[Lemma 7.3]{X1}, 
with $\phi_0 \coloneqq \phi \oplus (\rho \boxtimes S_{2\lambda-1})^{\oplus2}$, 
we have 
\begin{align*}
D_{\rho|\cdot|^\lambda}([I_{\sigma'}(\LL)])
&= \left[
\tau \rtimes \bigoplus_{i=1}^r (\pi(\phi_0,\ep'_{r,+}) \oplus \pi(\phi_0,\ep'_{r,-}))
\right]
\\&= 
\left[
\tau \rtimes \bigoplus_{i=1}^r (\Delta_\rho[\lambda-1,-(\lambda-1)] \rtimes \pi(\phi,\ep_r))
\right]
\\&= 
\left[
\Delta_\rho[\lambda-1,-(\lambda-1)] \times \tau \rtimes \bigoplus_{i=1}^r \pi(\phi,\ep_r) 
\right]
\\&= D_{\rho|\cdot|^\lambda}([I_{\sigma}(\LL)]), 
\end{align*}
as desired. 
\end{enumerate}
\par

Since the map $\sigma \mapsto \sigma'$ is an involution on 
\[
\{\sigma \in \Sy(\LL) \;|\; D_{\rho|\cdot|^\lambda}(I_\sigma(\LL)) \not= 0 \}
\]
satisfying $\sgn(\sigma') = -\sgn(\sigma)$ and 
$D_{\rho|\cdot|^{\lambda}}([I_\sigma(\LL)]) = D_{\rho|\cdot|^{\lambda}}([I_{\sigma'}(\LL)])$, 
we deduce that
\[
\sum_{\sigma \in \Sy(\LL)} \sgn(\sigma) D_{\rho|\cdot|^{\lambda}}([I_\sigma(\LL)]) = 0. 
\]
This completes the proof.
\end{proof}

\begin{lem}\label{lem2}
If $\lambda \in \R$ satisfies that $D_{\rho|\cdot|^\lambda}([\pi(\LL)]) \not= 0$, 
then 
\[
D_{\rho|\cdot|^\lambda}([\pi(\LL)]) = 
D_{\rho|\cdot|^{\lambda}} \circ p_\s \left(\sum_{\sigma \in \Sy(\LL)} \sgn(\sigma) [I_\sigma(\LL)] \right), 
\]
where $\s = \supp(\pi(\LL))$.
\end{lem}
\begin{proof}
By Proposition \ref{der}, we have
$\lambda \in \{x_1,\dots,x_t\}$ and $\lambda-1 \not\in \{x_1,\dots,x_t\}$.
Moreover, if we define $\LL' = (\ub{X'},\ub{l},\ub{\eta})$ from $\LL$ by 
replacing $X_\rho$ with 
\[
X'_\rho \coloneqq (X_\rho \setminus \{\rho|\cdot|^\lambda\}) \cup \{\rho|\cdot|^{\lambda-1}\}, 
\]
then $D_{\rho|\cdot|^\lambda}([\pi(\LL)]) = [\pi(\LL')]$. 
By Definition \ref{def:sgn}, we can see that 
$\Sy(\LL') = \Sy(\LL)$ and $D_{\rho|\cdot|^{\lambda}}([I_\sigma(\LL)]) = [I_\sigma(\LL')]$. 
Finally, if we set $\s' = \supp(\pi(\LL'))$, 
then $D_{\rho|\cdot|^\lambda} \circ p_\s = p_{\s'} \circ D_{\rho|\cdot|^{\lambda}}$.
Since we know the assertion of Theorem \ref{determinantal} for $\pi(\LL')$ 
by the induction hypothesis, 
we have 
\begin{align*}
D_{\rho|\cdot|^\lambda}([\pi(\LL)]) 
&= [\pi(\LL')]
\\&= p_{\s'} \left(\sum_{\sigma \in \Sy(\LL')} \sgn(\sigma) [I_\sigma(\LL')]\right)
\\&= p_{\s'} \circ D_{\rho|\cdot|^\lambda} \left(\sum_{\sigma \in \Sy(\LL)} \sgn(\sigma) [I_\sigma(\LL)]\right)
\\&= D_{\rho|\cdot|^{\lambda}} \circ p_\s \left(\sum_{\sigma \in \Sy(\LL)} \sgn(\sigma) [I_\sigma(\LL)] \right), 
\end{align*}
as desired. 
\end{proof}

Let $P = MN$ be the standard parabolic subgroup of $G_n$
such that $\pi(\LL)$ is an irreducible subquotient of $\Ind_P^{G_n}(\pi_M)$ 
for some $\pi_M \in \Cusp(M)$.
Then by applying Lemmas \ref{lem1} and \ref{lem2} repeatedly, 
we see that 
\[
\Jac_P([\pi(\LL)])
= \Jac_P\left(p_\s \left(\sum_{\sigma \in \Sy(\LL)} \sgn(\sigma) [I_\sigma(\LL)] \right)\right).
\]
Since $[\pi(\LL)]$ can be written as a linear combination of standard modules 
with the same cuspidal support as $\pi(\LL)$, 
this equation implies that 
\[
[\pi(\LL)] = p_\s \left(\sum_{\sigma \in \Sy(\LL)} \sgn(\sigma) [I_\sigma(\LL)] \right).
\]
This completes the proof of Theorem \ref{determinantal}.


\end{document}